\documentclass[11pt]{amsart}
\usepackage{amsmath,amssymb,xypic,array}
\usepackage[T1]{fontenc}
\usepackage{amsfonts}
\usepackage{amsmath}
\usepackage{amssymb}
\usepackage{amsthm}
\usepackage[cp850]{inputenc}
\usepackage{hyperref}
\usepackage{graphicx}
\usepackage{fancyhdr}
\usepackage{tikz}
\usepackage{longtable}
\usepackage{geometry}
\usepackage{textcomp}
\usetikzlibrary{trees}

\input{xy}
\xyoption{all}
\theoremstyle{theorem}

\newtheorem{Theorem}{Theorem}[section]
\newtheorem{theoremn}{Theorem}

\newtheorem{Lemma}[Theorem]{Lemma}
\newtheorem{Proposition}[Theorem]{Proposition}

\newtheorem{Corollary}[Theorem]{Corollary}
\newtheorem{Conjecture}[Theorem]{Conjecture}

\theoremstyle{definition}

\newtheorem{Definition}[Theorem]{Definition}
\newtheorem{Remark}[Theorem]{Remark}

\newtheorem{Notation}[Theorem]{Notation}

\newtheorem{Construction}[Theorem]{Construction}

\numberwithin{equation}{section}

\DeclareMathOperator{\NE}{NE}
\DeclareMathOperator{\Eff}{Eff}

\DeclareMathOperator{\coker}{coker}

\DeclareMathOperator{\Pic}{Pic}

\DeclareMathOperator{\Bir}{Bir}

\newcommand{\QED}{\ifhmode\unskip\nobreak\fi\quad {\rm Q.E.D.}} 

\newcommand{\Aut}{\operatorname{Aut}}

\newcommand{\quot}{/\hspace{-1.2mm}/}

\begin{document}
\title{On the biregular geometry of the Fulton-MacPherson compactification}

\author[Alex Massarenti]{Alex Massarenti}
\address{\sc Alex Massarenti\\
Universidade Federal Fluminense\\
Rua M\'ario Santos Braga\\
24020-140, Niter\'oi,  Rio de Janeiro\\ Brazil}
\email{alexmassarenti@id.uff.br}

\date{\today}

\maketitle

\begin{abstract}
Let $X[n]$ be the Fulton-MacPherson compactification of the configuration space of $n$ ordered points on a smooth projective variety $X$. We prove that if either $n\neq 2$ or $\dim(X)\geq 2$, then the connected component of the identity of $\Aut(X[n])$ is isomorphic to the connected component of the identity of $\Aut(X)$. When $X = C$ is a curve of genus $g(C)\neq 1$ we classify the dominant morphisms $C[n]\rightarrow C[r]$, and thanks to this we manage to compute the whole automorphism group of $C[n]$, namely $\Aut(C[n])\cong S_n\times \Aut(C)$ for any $n\neq 2$, while $\Aut(C[2])\cong S_2\ltimes (\Aut(C)\times \Aut(C))$. Furthermore, we extend these results on the automorphisms to the case where $X = C_1\times ... \times C_r$ is a product of curves of genus $g(C_i)\geq 2$. Finally, using the techniques developed to deal with Fulton-MacPherson spaces, we study the automorphism groups of some Kontsevich moduli spaces $\overline{M}_{0,n}(\mathbb{P}^N,d)$.  
\end{abstract}

\setcounter{tocdepth}{1}
\tableofcontents

\section*{Introduction}
The search for a natural compactification of the configuration space of $n$ distinct ordered points in a smooth algebraic variety $X$ has been a long-standing problem in algebraic geometry. In \cite{FM} W. Fulton and R. MacPherson constructed such compactification $X[n]$ from the Cartesian product $X^n$ by a sequence of blow-ups. For instance, when $n=2$ the Fulton-MacPherson compactification $X[2]$ is the blow-up of $X\times X$ along the diagonal, which has been the natural candidate since the nineteenth century. By \cite{FM} $X[n]$ is a smooth, irreducible variety.

Since then the compactification $X[n]$ has been widely studied and generalized by means of the theory of wonderful compactifications \cite{DP}, \cite{Li}, \cite{MP}. Indeed, $X[n]$ is the wonderful compactification of the set of all diagonals in $X^n$.

The Fulton-MacPherson compactification $C[n]$ of the configuration space of $n$ ordered points in a smooth projective curve $C$ is closely related to the Deligne-Mumford compactification $\overline{M}_{g,n}$ of the moduli space of smooth curves of genus $g$ with $n$-marked points.

Indeed, by \cite{KM11} $\overline{M}_{0,n}$ is the GIT quotient $\mathbb{P}^1[n]\quot PGL(2)$ with respect to a suitable linearization, and its birational geometry is closely related to the geometry of $\mathbb{P}^1[n]$ \cite[Section 3]{HK}, while if $g(C)\geq 3$ then $C[n]$ appears as the general fiber of the forgetful morphism $\overline{M}_{g,n}\rightarrow \overline{M}_g$. Furthermore, using the moduli spaces of weighted pointed curves introduced by B. Hassett by adding rational weights to the markings in \cite{Has}, Y. H. Kiem and H. B. Moon realized new compactifications of the configuration space of $n$ distinct points on $\mathbb{P}^1$ in \cite{KM11}. 

Moreover, $\mathbb{P}^1[n]$ is related to another important class of moduli spaces, namely the Kontsevich moduli spaces parametrizing stable maps introduced by M. Kontsevich in \cite{Ko}.

These spaces are denoted by $\overline{M}_{g,n}(X,\beta)$ where $X$ is a projective scheme and $\beta\in H_2(X,\mathbb{Z})$ is the homology class of a curve in $X$. A point in $\overline{M}_{g,n}(X,\beta)$ corresponds to a holomorphic map $\alpha$ from an $n$-pointed genus $g$ curve $C$ to $X$ such that $\alpha_{*}([C])=\beta$. If $X$ is a homogeneous variety then there exists a smooth, irreducible Deligne-Mumford stack $\overline{\mathcal{M}}_{0,n}(X,\beta)$ whose coarse moduli space is $\overline{M}_{0,n}(X,\beta)$ \cite{FP}.

We will deal mainly with the case $X = \mathbb{P}^N$, the class $\beta$ is then completely determined by its degree and we will write $\beta = d[L]$, where $[L]$ is the class of a line in $\mathbb{P}^N$. The space $\overline{M}_{0,n}(\mathbb{P}^{N},d)$ admits $n$ evaluation maps $ev_i:\overline{M}_{0,n}(\mathbb{P}^{N},d)\rightarrow \mathbb{P}^N$ associating to a stable map its value on the $i$-th marked point. Furthermore, when $d = N = 1$ the birational morphism $ev_1\times ...\times ev_n:\overline{M}_{0,n}(\mathbb{P}^1,1)\rightarrow(\mathbb{P}^1)^n$ realizes $\overline{M}_{0,n}(\mathbb{P}^1,1)$ as the Fulton-MacPherson compactification $\mathbb{P}^1[n]$.

The Picard group of $\overline{M}_{0,n}(\mathbb{P}^N,d)$ was computed by R. Pandharipande in \cite{Pa99}, and the cones of divisors of $\overline{M}_{0,n}(\mathbb{P}^N,d)$ have been carefully analyzed by I. Coskun, J. Harris and J. Starr in \cite{CHS}. In Section \ref{LinPenc}, thanks to their description of $\Pic(\mathbb{P}^1[n])$, we manage to classify base point free pencils on $\mathbb{P}^1[n]\cong \overline{M}_{0,n}(\mathbb{P}^1,1)$. The first step of our argument consists in analyzing the birational and the projective geometry of $\mathbb{P}^1[3]$. This is a Fano variety obtained by blowing-up the small diagonal in $\mathbb{P}^1\times\mathbb{P}^1\times\mathbb{P}^1$, or equivalently three skew lines in $\mathbb{P}^3$.

In particular, it is a Mori Dream Space so that its effective and nef cones are polyhedral, and we manage to compute its Mori cone. On the other hand, by \cite[Corollary 1.4]{CT15} we know that $\overline{M}_{0,n}$ is not a Mori Dream Space for $n$ big enough. Since as soon as $n\geq 3$ there is a dominant morphism $\rho:\mathbb{P}^1[n]\rightarrow\overline{M}_{0,n}$, forgetting the map, \cite[Theorem 1.1]{Ok16} yields that $\mathbb{P}^1[n]$ is not a Mori Dream Space for $n$ sufficiently big as well.\\
Nevertheless, the analysis of base point free pencils on $\mathbb{P}^1[3]$ will be the first step of an inductive argument that will lead us to the classification of dominant morphisms $\mathbb{P}^1[n]\rightarrow\mathbb{P}^{1}[r]$ in Sections \ref{LinPenc} and \ref{sectionFib}. 

There are several natural morphisms $\mathbb{P}^1[n]\rightarrow\mathbb{P}^1$, namely the evaluation morphisms mentioned above, the forgetful morphisms $\pi_I:\mathbb{P}^1[n]\rightarrow \mathbb{P}^1[1]\cong\mathbb{P}^1$ forgetting the $n-1$ points labeled by the set $I$, and the morphisms $\pi_{J}\circ\rho:\mathbb{P}^1[n]\rightarrow\overline{M}_{0,4}\cong\mathbb{P}^1$, where $\rho:\mathbb{P}^1[n]\rightarrow\overline{M}_{0,n}$ forgets the map, and $\pi_{J}:\overline{M}_{0,n}\rightarrow\overline{M}_{0,4}$ forgets the $n-4$ points labeled by the set $J$. Note that by Remark \ref{forgeval} the evaluation map $ev_i$ may be identified with the forgetful morphism $\pi_{I}$, with $I = \{1,...,n\}\setminus \{i\}$. We call modular base point free pencils the linear systems associated to these morphisms, see Definition \ref{modmor}.\\ 
Furthermore, if $X$ is an arbitrary smooth variety there are forgetful morphisms $\pi_I:X[n]\rightarrow X[r]$ forgetting $n-r$ of the points, these are just the liftings of the projections $X^n\rightarrow X^r$ to the blow-ups.\\ 
The main results on fibrations in Propositions \ref{propfact}, \ref{r=2}, \ref{propfactg2}, Theorem \ref{thfact}, and Corollary \ref{corfactprod} can be summarized with the following statement:

\begin{theoremn}\label{th1}
Let $\psi:\mathbb{P}^1[n]\rightarrow\mathbb{P}^{1}[r_1]\times ...\times\mathbb{P}^{1}[r_k]$ be a dominant morphism. If $r_i\geq 3$ for some $i=1,...,k$ we assume in addition that $\psi$ has connected fibers. Then $\psi$ factors through a product of forgetful morphisms of type $\pi_I$ and $\pi_{J}\circ\rho$. Finally, if $r_i\geq 3$ for any $i = 1,...,k$ then $\psi$ factors through a product of forgetful morphisms of type $\pi_I$ only.\\
Let $C$ be a smooth projective curve of genus $g(C)\geq 2$, and $\psi:C[n]\rightarrow C[r_1]\times ...\times C[r_k]$ be a dominant morphism. As before, if $r_i\geq 3$ for some $i=1,...,k$ we assume in addition that $\psi$ has connected fibers. Then $\psi$ factors through a product of forgetful morphisms.
\end{theoremn}

Note that for curves of genus one, or more generally for abelian varieties Theorem \ref{th1} does not hold. Indeed, if $A$ is an abelian variety the multiplication map $A\times A\rightarrow A$ does not factor through one of the projections.

In Section \ref{autsect} we study the automorphism groups of $X[n]$ and of some Kontsevich moduli spaces $\overline{M}_{0,n}(\mathbb{P}^N,d)$. In several cases, automorphisms of moduli spaces tend to be modular, in the sense that they can be described in terms of the objects parametrized by the moduli spaces themselves. Questions of this type have been addressed by several authors in a series of papers, see for instance \cite{BM}, \cite{Ma}, \cite{Lin04}, \cite{Lin11}, \cite{Ro71}, \cite{GKM}, \cite{KMc}, \cite{Far} for moduli spaces of pointed curves, \cite{AP16} for the tropical counterpart of \cite{BM}, \cite{BGM} for moduli spaces of vector bundles over a curve, and \cite{BM16} for generalized quot schemes. We confirm this behavior also for Fulton-MacPherson, and for some Kontsevich moduli spaces.

Thanks to a result due to M. Brion \cite{Br} in the algebraic setting, and to A. Blanchard \cite{Bl} in the analytic setting we compute the connected component of the identity of $\Aut(X[n])$. Furthermore, as an application of Theorem \ref{th1} we manage to control the discrete part of $\Aut(C[n])$. The main results on the automorphism groups in Propositions \ref{connFM}, \ref{druel2}, and Theorem \ref{autFMP1} may be summarized as follows:

\begin{theoremn}\label{th2}
Let $X[n]$ be the Fulton-MacPherson compactification of the configuration space of $n$ ordered points on a smooth projective variety $X$. If either $n\neq 2$ or $\dim(X)\geq 2$, then the connected component of the identity of $\Aut(X[n])$ is isomorphic to the connected component of the identity of $\Aut(X)$, that is
$$\Aut^{o}(X[n])\cong \Aut^{o}(X)$$
for any $n$, and if $X = C$ is a curve then $\Aut^{o}(C[2])\cong \Aut^{o}(C)\times \Aut^{o}(C)$.\\
Furthermore, if $X = C$ is a curve with $g(C)\neq 1$ then we have 
$$\Aut(C[n])\cong S_n\times \Aut(C)$$
if $n\neq 2$, while $\Aut(C[2])\cong S_2\ltimes (\Aut(C)\times \Aut(C))$.
\end{theoremn}
Note that Theorem \ref{th2} does not hold if $C$ has genus one. For instance, in this case by Remark \ref{GL2} the group $GL(2,\mathbb{Z})$ of matrices with integers entries and determinant plus or minus one acts on $C\times C$.\\ 
In Corollary \ref{mgn}, thanks to Theorem \ref{th2}, we get a simple proof of the main result on the automorphisms of $\overline{M}_{g,n}$ in \cite{Ma} when $g\geq 3$. 

Furthermore, in Proposition \ref{druel1} we extend these techniques to the case when $X = C_1\times ...\times C_r$ is a product of curves of genus $g(C_i)\geq 2$, and to some Kontsevich moduli spaces and moduli stacks. The results on products of curves in Lemma \ref{autcp} and Proposition \ref{druel1} can be summarized as follows:

\begin{theoremn}\label{th4}
Let $X = C_1\times ...\times C_r$ be a product of curves with $g(C_i)\geq 2$ for any $i = 1,...,r$, and let $[C_{r_1}],...,[C_{r_k}]$ be the isomorphism classes of curves in $\{C_1,...,C_r\}$, where $r_i$ is the number of curves of class $[C_{r_i}]$. If $n\neq 2$ then
$$\Aut(X[n])\cong S_n\times((S_{r_1}\ltimes \Aut(C_{r_1})^{r_1})\times ...\times (S_{r_k}\ltimes \Aut(C_{r_k})^{r_k}))\cong S_n\times \Aut(X)$$
while if $n = 2$ and $r\geq 2$ we have
$$\Aut(X[2])\cong S_2^r\ltimes ((S_{r_1}\ltimes \Aut(C_{r_1})^{r_1})\times ...\times (S_{r_k}\ltimes \Aut(C_{r_k})^{r_k})) \cong S_2^r\ltimes \Aut(X)$$
Finally, if $n = 2$ and $r=1$ then $X = C_1$, and $\Aut(C_1[2])\cong S_2\ltimes(\Aut(C_1)\times \Aut(C_1))$.
\end{theoremn}

In the following we summarize the results in this direction for Kontsevich spaces in Propositions \ref{propGrass}, \ref{connRNC}, and Corollary \ref{corRNC}.

\begin{theoremn}\label{th3}
Let $\overline{\mathcal{M}}_{0,n}(\mathbb{P}^N,d)$ be the Deligne-Mumford stack of stable maps of degree $d$ from an $n$-pointed rational curve to $\mathbb{P}^N$, and let $\overline{M}_{0,n}(\mathbb{P}^N,d)$ be its coarse moduli space. If $N\geq 2$ and $n\geq 1$ then
$$\Aut^{o}(\overline{M}_{0,n}(\mathbb{P}^N,1))\cong PGL(2)\times PGL(N+1)$$
for any $n\neq 2$, and $\Aut^{o}(\overline{M}_{0,2}(\mathbb{P}^N,1))\cong PGL(2)\times PGL(2)\times PGL(N+1)$.\\
Furthermore, for the Kontsevich spaces parametrizing rational normal curves we have:
$$\Aut^{o}(\overline{M}_{0,k}(\mathbb{P}^n,n))\cong\Aut^{o}(\overline{\mathcal{M}}_{0,k}(\mathbb{P}^n,n))\cong PGL(n+1)$$
for any $n\geq 3$, and $k\geq n+2$. 
\end{theoremn}
In this paper we work over an algebraically closed field of characteristic zero. However, by \cite[Lemma 1.2]{FaM} the results in Theorems \ref{th2}, \ref{th4} and \ref{th3} can be easily extended to any field, not necessarily an algebraically closed one, of characteristic zero. Finally, in Conjecture \ref{conj} we propose a conjecture on $\Aut(X[n])$ when $X$ is of general type.

\subsection*{Acknowledgments}
I thank St\'ephane Druel for his useful comments and suggestions, particularly about Propositions \ref{druel2} and \ref{druel1}, that helped me to improve a preliminary version of the paper.

Furthermore, I would like to thank the referees for the careful reading that helped me to improve the exposition, in particular one of them for pointing out some inaccuracies in the proof of Proposition \ref{propfact} and for suggesting me how to fix them.

The author is a member of the Gruppo Nazionale per le Strutture Algebriche, Geometriche e le loro Applicazioni of the Istituto Nazionale di Alta Matematica "F. Severi" (GNSAGA-INDAM).

\section{Fulton-MacPherson compactification and Kontsevich moduli spaces}\label{FMK}
All through the paper we work over an algebraically closed field of characteristic zero. In \cite{FM} W. Fulton and R. MacPherson constructed a natural compactification of the configuration space of $n$ distinct ordered points in a smooth algebraic variety $X$. The configuration space
$$\mathcal{C}(X,n) = X^{n}\setminus \bigcup_{1\leq i,j\leq n}\Delta_{i,j}$$
is the complement in the Cartesian product $X^n$ of the diagonals $\Delta_{i,j} = \{(x_1,...,x_n)\in X^n \: |\: x_i =x_j\}$. The Fulton-MacPherson compactification $X[n]$ of $\mathcal{C}(X,n)$ can be realized from $X^n$ via a sequence of blow-ups along smooth and irreducible centers, \cite[Section 3]{FM}.

\subsection{Fulton-MacPherson blow-up construction of $X[n]$}\label{FMConst} 
Let us recall the construction of $X[n]$ given in \cite[Section 3]{FM}. If $n = 1$ then $X[1]=X$, and if $n = 2$ we have that $X[2]$ is the blow-up $\pi_2:X[2]\rightarrow X^2$ of $X^2$ along $\Delta_{1,2}$. Now, assuming that $\pi_{n-1}:X[n-1]\rightarrow X^{n-1}$ has already been constructed, $X[n]$ may be realized in the following way.
\begin{Construction}
For any $S = \{i_1,...,i_s\}\subset\{1,...,n-1\}$ let $\Delta_S = \{(x_1,...,x_{n-1})\in X^{n-1}\:|\: x_{i_1} = ... = x_{i_s}\}$, and let $E_S\subset X[n-1]$ be the exceptional divisor over $\Delta_S$. Finally, we denote by $p_i:X^{n-1}\rightarrow X$ the projections.
\begin{itemize}
\item[-] Let $\widetilde{E}_{1,...,n-1} = \{(y,x)\in X[n-1]\times X\: |\: (p_i\circ\pi_{n-1})(y) = x\: \forall\: i = 1,...,n-1 \}\subset X[n-1]\times X$. Let $X[n-1]_1$ be the blow-up of $X[n-1]\times X$ along $\widetilde{E}_{1,...,n-1}$.
\item[-] For any $S\subset\{1,...,n-1\}$ with $|S| = n-2$, let 
$$\widetilde{E}_{S} = \{(y,x)\in X[n]\times X\: |\: (p_i\circ\pi_{n-1})(y) = x \: \forall\: i\in S\}\subset X[n-1]\times X$$
Note that since $\widetilde{E}_{1,...,n-1}$ has been blown-up in the preceding step the strict transforms in $X[n-1]_1$ of the $\widetilde{E}_{S}$'s do not intersect. Let $X[n-1]_2$ be the blow-up of $X[n-1]_1$ along the strict transforms of the $\widetilde{E}_{S}$'s. Note that the image of the exceptional divisor over the strict transform of $\widetilde{E}_{S}$ via the projection $X[n-1]_2\rightarrow X^{n}$ is the diagonal $\Delta_{S\cup \{n\}}$.
\item[-] We repeat recursively the construction in the preceding step for any subset $S\subset\{1,...,n-1\}$ with $2\leq |S|\leq n-3$ in order of decreasing cardinality of the set $S$, and we denote by $X[n-1]_u$ the variety obtained by this sequence of blow-ups.
\item[-] For any $i = 1,...,n-1$ we consider
$$\widetilde{X[n-1]}_i = \{(y,x)\in X[n-1]\times X \: |\: (p_i\circ \pi_{n-1})(y) = x\}\subset X[n-1]\times X$$
The projection $X[n-1]\times X\rightarrow X^{n}$ maps $\widetilde{X[n]}_i$ onto the diagonal $\Delta_{i,n}$. In order to get $X[n]$ we blow-up $X[n-1]_u$ along the strict transforms of the $\widetilde{X[n-1]}_i$'s.
\end{itemize}
Finally, we denote by $f_n:X[n]\rightarrow X^n$ the composition of these blow-ups.
\end{Construction}

\begin{Lemma}\label{picnum}
The Picard number of $X[n]$ is given by $\rho(X[n]) = \rho(X^n)+2^n-n-1$ if $\dim(X)\geq 2$, and by $\rho(C[n]) = \rho(C^n)+2^n-\frac{n}{2}(n+1)-1$ if $X = C$ is a curve.
\end{Lemma}
\begin{proof}
It is enough to prove that $X[n]$ is obtained from $X^n$ by a sequence of $2^n-n-1$ blow-ups. This is clear for $n = 1$. We proceed by induction on $n$. By induction hypothesis $X[n-1]\times X$ is obtained from $X^n = X^{n-1}\times X$ via a sequence of $2^{n-1}-n$ blow-ups. To conclude, it is enough to observe that by Construction \ref{FMConst} $X[n]$ is obtained from $X[n-1]\times X$ by a sequence of 
$$\sum_{k=1}^{n-1}\binom{n-1}{k} = \sum_{k=0}^{n-1}\binom{n-1}{k}-1 = 2^{n-1}-1$$  
blow-ups. If $X = C$ is a curve the last $\binom{n}{2}$ blow-ups are blow-ups of Cartier divisors, therefore they do not modify the variety, and in particular they do not contribute to the Picard number.
\end{proof}

A symmetric construction of $X[n]$ has been realized by several authors \cite{DP}, \cite{Li}, \cite{MP} as follows.

\begin{Proposition}\label{sym}
Let us consider the following sequence of blow-ups:
\begin{itemize}
\item[-] $X[n]_1$ is the blow-up of $X[n]$ along $\Delta_{1,...,n}$;
\item[-] $X[n]_2$ is the blow-up of $X[n]_1$ along the strict transforms of the diagonals $\Delta_S$ with $|S| = n-1$;\\
\vdots
\item[-] $X[n]_i$ is the blow-up of $X[n]_{i-1}$ along the strict transforms of the diagonals $\Delta_S$ with $|S| = n-i+1$;\\
\vdots
\item[-] $X[n]_{n-1}$ is the blow-up of $X[n]_{n-2}$ along the strict transforms of the diagonals $\Delta_S$ with $|S| = 2$.
\end{itemize}
Let $g_n:X[n]_{n-1}\rightarrow X^n$ be the birational morphism given by the above sequence of blow-ups. Then $X[n]_{n-1}$ is a smooth variety isomorphic to $X[n]$.
\end{Proposition}
Indeed, Proposition \ref{sym} follows from \cite[Theorem 1.3]{Li}. The reader may also see \cite[Section 4.2]{Li} for further details.

\subsection{Kontsevich moduli spaces of stable maps to homogeneous varieties}
An $n$-pointed rational pre-stable curve $(C,(x_{1},...,x_{n}))$ is a projective, connected, reduced rational curve with at most nodal singularities of arithmetic genus zero, with $n$ distinct and smooth marked points $x_1,...,x_n\in C$. We will refer to the marked and the singular points of $C$ as special points.\\
Let $X$ be a homogeneous variety. A map $(C,(x_{1},...,x_{n}),\alpha)$, where $\alpha:C\rightarrow X$ is a morphism from an $n$-pointed rational pre-stable curve to $X$ is stable if any component $E\cong\mathbb{P}^{1}$ of $C$ contracted by $\alpha$ contains at least three special points.\\
Now, let us fix a class $\beta\in H_2(X,\mathbb{Z})$. By \cite[Theorem 2]{FP} there exists a smooth, proper, and separated Deligne-Mumford stack $\overline{\mathcal{M}}_{0,n}(X,\beta)$ parametrizing isomorphism classes of stable maps $[C,(x_{1},...,x_{n}),\alpha]$ such that $\alpha_{*}[C] = \beta$.\\
Furthermore, by \cite[Corollary 1]{KP} the coarse moduli space $\overline{M}_{0,n}(X,\beta)$ associated to the stack $\overline{\mathcal{M}}_{0,n}(X,\beta)$ is a normal, irreducible, projective variety with at most finite quotient singularities of dimension
$$
\dim(\overline{M}_{0,n}(X,\beta)) = \dim(X)+\int_{\beta}c_1(T_X)+n-3
$$
The variety $\overline{M}_{0,n}(X,\beta)$ is called the \textit{moduli space of stable maps}, or the \textit{Kontsevich moduli space} of stable maps of class $\beta$ from a rational pre-stable $n$-pointed curve to $X$. The boundary $\partial\overline{M}_{0,n}(X,\beta) = \overline{M}_{0,n}(X,\beta)\setminus M_{0,n}(X,\beta)$ is a simple normal crossing divisor in $\overline{M}_{0,n}(X,\beta)$, whose points parametrize isomorphism classes of stable maps $[C,(x_{1},...,x_{n}),\alpha]$ where $C$ is a reducible curve. When $X = \mathbb{P}^N$, we will write $\overline{M}_{0,n}(\mathbb{P}^N,d)$ for $\overline{M}_{0,n}(\mathbb{P}^N,d[L])$, where $L\subseteq\mathbb{P}^N$ is a line.

\subsubsection{Natural morphisms}\label{natmorp}
There are several natural morphisms defined on $\overline{M}_{0,n}(X,\beta)$. For any $i\in\{1,...,n\}$ we have the $i$-th \textit{evaluation map}:
$$
\begin{array}{cccc}
ev_i: &\overline{M}_{0,n}(X,\beta)& \longrightarrow & X\\
      & \left[C,(x_1,...,x_n),\alpha\right] & \longmapsto & \alpha(x_i)
\end{array}
$$
Furthermore, there are the \textit{forgetful morphisms}
\stepcounter{Theorem}
\begin{equation}\label{forg1}
\pi_{i_1,...,i_k}:\overline{M}_{0,n}(X,\beta)\rightarrow \overline{M}_{0,n-k}(X,\beta)
\end{equation}
forgetting the marked points $x_{i_1},...,x_{i_k}$ and stabilizing the resulting map, and if $n\geq3$ we have the forgetful morphism
$$\rho:\overline{M}_{0,n}(X,\beta)\rightarrow \overline{M}_{0,n}$$
forgetting the map and stabilizing the domain curve.

\begin{Remark}\label{kfm}
If $d = N = 1$ the Kontsevich moduli space $\overline{M}_{0,n}(\mathbb{P}^1,1)$ is isomorphic to the Fulton-MacPherson compactification $\mathbb{P}^1[n]$. By Section \ref{natmorp} we have a morphism
$$
\begin{array}{cccc}
ev := ev_1\times ...\times ev_n: &\overline{M}_{0,n}(\mathbb{P}^1,1)& \longrightarrow & (\mathbb{P}^{1})^n\\
      & \left[C,(x_1,...,x_n),\alpha\right] & \longmapsto & (\alpha(x_1),...,\alpha(x_n))
\end{array}
$$
For any $\{i_1,...,i_k\}\subset\{1,...,n\}$ with $k\geq 2$ let $D_{i_1,...,i_k}\subset \overline{M}_{0,n}(\mathbb{P}^1,1)$ be the divisor whose general point corresponds to a stable map $[C,(x_1,...,x_n),\alpha]$, where $C = C_1\cup C_2$ is the union of two smooth rational curves, $C_1$ has marked points $x_{i_1},...,x_{i_k}$ and is contracted to a point via $\alpha$, while $C_2$ is mapped isomorphically onto $\mathbb{P}^1$ by $\alpha$. Then $ev(D_{i_1,...,i_k}) = \Delta_{i_1,...,i_k}\subset (\mathbb{P}^{1})^n$, and $ev$ is exactly the blow-up morphism $g_n:\mathbb{P}^1[n]\rightarrow(\mathbb{P}^1)^n$ in Proposition \ref{sym}.
\end{Remark}

We will need the following simple result on the fibers of the evaluation maps.

\begin{Lemma}\label{isofib}
Let $X$ be a homogeneous variety. Then all the fibers of the evaluation map $ev_i:\overline{M}_{0,n}(X,\beta)\rightarrow X$ are isomorphic.
\end{Lemma}
\begin{proof}
Let $p,q\in X$ be two points, and let $F_p = ev_i^{-1}(p)$, $F_q = ev_i^{-1}(q)$ be the corresponding two fibers of $ev_i$. Let $\mu\in \Aut^{o}(X)$ be an automorphism of $X$ such that $\mu(p) = q$. Since $\mu$ is in the connected component of the identity of $\Aut(X)$ it must preserve the class $\beta$, and  
$$
\begin{array}{cccc}
f_{\mu} &F_p& \longrightarrow & F_q\\
      & \left[C,(x_1,...,x_n),\alpha\right] & \longmapsto & \left[C,(x_1,...,x_n),\mu\circ\alpha\right]
\end{array}
$$
is an isomorphism whose inverse is $f_{\mu^{-1}}$.
\end{proof}

\section{Base point free pencils on $\mathbb{P}^1[n]$}\label{LinPenc}
The main purpose of this section is to classify base point free pencils on the Fulton-MacPherson compactification $\mathbb{P}^1[n]$ of the configuration space of $n$ points on $\mathbb{P}^1$. We begin by describing fibrations of the Cartesian product of smooth curves of genus different from one.

\begin{Lemma}\label{fibg2}
Let $C_1,...,C_n$, $B_1,...,B_r$ be smooth projective curves with either $g(B_i)\geq 2$ for any $i = 1,...,r$ or $g(C_i)= 0$ for any $i=1,...,n$, and let $\psi:C_1\times ...\times C_n\rightarrow B_1\times ...\times B_r$ be a dominant morphism. Then there exist $i_1,...,i_r\in\{1,...,n\}$, and morphisms $f_{i_j}:C_{i_j}\rightarrow B_j$ such that the following diagram
  \[
  \begin{tikzpicture}[xscale=4.8,yscale=-1.8]
    \node (A0_0) at (0, 0) {$C_1\times ...\times C_n$};
    \node (A1_0) at (0, 1) {$C_{i_1}\times ...\times C_{i_r}$};
    \node (A1_1) at (1, 1) {$B_1\times ...\times B_r$};
    \path (A0_0) edge [->,swap]node [auto] {$\scriptstyle{\pi_{i_1}\times ...\times\pi_{i_r}}$} (A1_0);
    \path (A1_0) edge [->]node [auto] {$\scriptstyle{f_{i_1}\times ...\times f_{i_r}}$} (A1_1);
    \path (A0_0) edge [->]node [auto] {$\scriptstyle{\psi}$} (A1_1);
  \end{tikzpicture}
  \]
commutes, where $\pi_{i_j}:C_1\times ...\times C_{n}\rightarrow C_{i_j}$ is the $i_j$-th canonical projection. In particular, any dominant morphism $\phi:C^n\rightarrow C$ with $g(C)\neq 1$ factors through one of the projections.
\end{Lemma}
\begin{proof}
It is enough to prove that for any projection $p_i:B_1\times ...\times B_r\rightarrow B_i$ the morphism $p_i\circ\psi$ factors through some projection $\pi_{j_i}:C_1\times ...\times C_n\rightarrow C_{j_i}$.\\
If $g(B_i)\geq 2$ the result follows from \cite[Lemma 3.8]{Ca00}. Indeed, the statement of \cite[Lemma 3.8]{Ca00} concerns products of two curves but its proof works identically for a product of an arbitrary number of curves.\\
Now, let us consider the case of $X=\mathbb{P}^{1}_{1} \times ...\times \mathbb{P}^{1}_{n}$. Let $l_i$ be the class of the $i$-th factor of the product, and $h_i = \pi_i^{*}\mathcal{O}_{\mathbb{P}^1_{i}}(1)$. Let us recall that the Chow ring of $X$ is given by 
$$A^{*}(X) \cong \mathbb{Z}[h_1,...,h_n]/(h_1^2,...,h_n^2).$$
Furthermore $l_i\cdot h_j = \delta_{ij}$, where $\delta_{ij}$ is the Kronecker delta. Since $X$ is a toric variety its Mori cone $\NE(X)$ is the polyhedral cone generated by classes of the torus invariant curves $l_i$ for $i = 1,...,n$. Moreover, since $\Pic(X)$ is generated by $h_1,...,h_n$ we may write a divisor $D$ on $X$ as $D \equiv a_1h_1+...+a_nh_n$, and $D\cdot l_i = a_i$ yields that $D$ is nef if and only if $a_i\geq 0$ for $i=1,...,n$.\\
Now, a dominant morphism $\phi:X\rightarrow\mathbb{P}^1$ is induced by a nef divisor $D \equiv a_1h_1+...+a_nh_n$ on $X$. Let $p,q\in\mathbb{P}^1$ be two distinct points in $\mathbb{P}^1$, then the fibers $D_p = \phi^{-1}(p), D_q = \phi^{-1}(q)$ are elements of the linear system of $D$, that is they are linearly equivalent to $D \equiv a_1h_1+...+a_nh_n$. Now, since $p\neq q$ and keeping in mind that $h_i^2=0$ for any $i$ we get that 
$$D_p\cdot D_q = \sum_{i=1}^{n}a_i^2h_i^2+2\sum_{i<j}a_ia_jh_ih_j =2\sum_{i<j}a_ia_jh_ih_j =0$$
in the group $A^2(X)$ of classes of cycles of codimension two in $X$. Hence $n-1$ of the $a_i$'s must be zero. Let $a_j$ be the only non-zero coefficient in the expression of $D$. Then $D\equiv a_jh_j$, and $\phi$ factors through the projection $\pi_j:X\rightarrow \mathbb{P}^1_{j}$.
\end{proof}

\begin{Remark}
Note that Lemma \ref{fibg2} does not hold for curves of genus one, or more generally for abelian varieties. Indeed, if $A$ is an abelian variety the multiplication map $A\times A\rightarrow A$ does not factor through one of the projections.
\end{Remark}

For the rest of this section we will focus on dominant morphisms $\mathbb{P}^1[n]\rightarrow\mathbb{P}^1$. Recall that by Remark \ref{kfm} we can identify $\mathbb{P}^1[n]$ with the Kontsevich moduli space $\overline{M}_{0,n}(\mathbb{P}^1,1)$.

\subsection{The Picard group of $\mathbb{P}^1[n]$}\label{cones}
We summarize the results on the Picard group of $\overline{M}_{0,n}(\mathbb{P}^N,d)$ in \cite{CHS} for the particular case $\overline{M}_{0,n}(\mathbb{P}^1,1)\cong\mathbb{P}^1[n]$. By \cite[Section 2.1]{CHS} there is a morphism
$$f_p:\overline{M}_{0,n+1}\rightarrow \mathbb{P}^1[n]$$
defined as follows. Let us fix a point $p\in\mathbb{P}^1$. For any $[C,(x_1,...,x_{n+1})]\in\overline{M}_{0,n+1}$ we attach a $\mathbb{P}^1$ at $x_{n+1}$, and consider the morphism $\alpha:C\cup\mathbb{P}^1\rightarrow\mathbb{P}^1$ that contracts $C$ to $p\in\mathbb{P}^1$ and maps the added rational tail isomorphically to $\mathbb{P}^1$.
$$
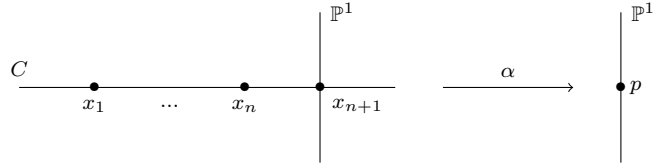
\begin{tikzpicture}
\draw  plot[smooth, tension=.7] coordinates {(-5,0) (-10,0)};
\draw  plot[smooth, tension=.7] coordinates {(-6,1) (-6,-1)};
\draw  plot[smooth, tension=.7] coordinates {(-2,1) (-2,-1)};
\node at (-10,0.25) {$\scriptstyle{C}$};
\node at (-9,-0.25) {$\scriptstyle{x_1}$};
\node at (-9,0) {$\scriptstyle{\bullet}$};
\node at (-8,-0.25) {$\scriptstyle{...}$};
\node at (-7,-0.25) {$\scriptstyle{x_{n}}$};
\node at (-7,0) {$\scriptstyle{\bullet}$};
\node at (-5.5,-0.25) {$\scriptstyle{x_{n+1}}$};
\node at (-5.7,1) {$\scriptstyle{\mathbb{P}^1}$};
\node at (-1.7,1) {$\scriptstyle{\mathbb{P}^1}$};
\node (A) at (-4.5,0) {};
\node (B) at (-2.5,0) {};
\path (A) edge [->]node [auto] {$\scriptstyle{\alpha}$} (B);
\node at (-1.8,0) {$\scriptstyle{p}$};
\node at (-6,0) {$\scriptstyle{\bullet}$};
\node at (-2,0) {$\scriptstyle{\bullet}$};
\end{tikzpicture}
$$
Furthermore, for any $i=1,...,n$ there is a morphism
$$g_i:\mathbb{P}^1\rightarrow \mathbb{P}^1[n]$$
defined as follows. Let us fix an $(n-1)$-pointed rational curve $C$. At a general point of $C$ we attach a $\mathbb{P}^1$ with the marked point $x_i\in\mathbb{P}^1$. The domain of the stable map is $C\cup\mathbb{P}^1$, and the map $\alpha:C\cup\mathbb{P}^1\rightarrow\mathbb{P}^1$ is the identity on $\mathbb{P}^1$, and contracts $C$. 
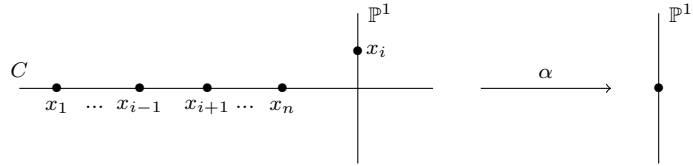
$$
\begin{tikzpicture}
\draw  plot[smooth, tension=.7] coordinates {(-5,0) (-10.5,0)};
\draw  plot[smooth, tension=.7] coordinates {(-6,1) (-6,-1)};
\draw  plot[smooth, tension=.7] coordinates {(-2,1) (-2,-1)};
\node at (-10.5,0.25) {$\scriptstyle{C}$};
\node at (-10,-0.25) {$\scriptstyle{x_1}$};
\node at (-10,0) {$\scriptstyle{\bullet}$};
\node at (-9.5,-0.25) {$\scriptstyle{...}$};
\node at (-8.9,-0.25) {$\scriptstyle{x_{i-1}}$};
\node at (-8.9,0) {$\scriptstyle{\bullet}$};
\node at (-8,-0.25) {$\scriptstyle{x_{i+1}}$};
\node at (-8,0) {$\scriptstyle{\bullet}$};
\node at (-7.5,-0.25) {$\scriptstyle{...}$};
\node at (-7,-0.25) {$\scriptstyle{x_{n}}$};
\node at (-7,0) {$\scriptstyle{\bullet}$};
\node at (-5.7,1) {$\scriptstyle{\mathbb{P}^1}$};
\node at (-1.7,1) {$\scriptstyle{\mathbb{P}^1}$};
\node (A) at (-4.5,0) {};
\node (B) at (-2.5,0) {};
\path (A) edge [->]node [auto] {$\scriptstyle{\alpha}$} (B);
\node at (-2,0) {$\scriptstyle{\bullet}$};
\node at (-5.75,0.5) {$\scriptstyle{x_i}$};
\node at (-6,0.5) {$\scriptstyle{\bullet}$};
\end{tikzpicture}
$$
Varying the point $x_i\in\mathbb{P}^1$ we get the morphism $g_i:\mathbb{P}^1\rightarrow \mathbb{P}^1[n]$. By \cite[Theorem 2.3]{CHS} we have that the map
\stepcounter{Theorem}
\begin{equation}\label{maph}
h : = f_p^{*}\times g_{1}^{*}\times ...\times g_n^{*}:\Pic(\mathbb{P}^1[n])\rightarrow \Pic(\overline{M}_{0,n+1})\times \Pic(\mathbb{P}^1)^n
\end{equation}
is an isomorphism. Furthermore the image via $h$ of the ample, nef and eventually-free cone of $\mathbb{P}^1[n]$ is the product of the ample, nef and eventually-free cones respectively of $\overline{M}_{0,n+1}$ and of the $\mathbb{P}^1$ factors.

\begin{Notation}\label{not}
The map $h = f_p^{*}\times g_1^{*}\times ...\times g_n^{*}$ in (\ref{maph}) defines an isomorphism between $\Pic(\mathbb{P}^1[n])$ and $\Pic(\overline{M}_{0,n+1})\times\Pic(\mathbb{P}^1)^n$. Therefore, we may write any divisor $D$ in $\mathbb{P}^1[n]$ as $D\equiv D_{0,n+1}+a_1H_1+...+a_nH_n$, where $D_{0,n+1}$ is a divisor on $\overline{M}_{0,n+1}$, the $H_i$'s are generators of the factors $\Pic(\mathbb{P}^1)$, and $a_i\in\mathbb{Z}$. 
\end{Notation}

\subsection{Modular linear pencils on $\overline{M}_{0,n}(\mathbb{P}^N,d)$}
By Section \ref{natmorp} there are several natural morphisms from $\overline{M}_{0,n}(\mathbb{P}^N,d)$ onto $\mathbb{P}^1$. We may consider the composition 
\stepcounter{Theorem}
\begin{equation}\label{forg}
\overline{M}_{0,n}(\mathbb{P}^N,d)\xrightarrow{\makebox[1.5cm]{$\rho$}}\overline{M}_{0,n}\xrightarrow{\pi_{i_1,...,i_{n-4}}}\overline{M}_{0,4}\cong\mathbb{P}^1
\end{equation}
where $\pi_{i_1,...,i_{n-4}}:\overline{M}_{0,n}\rightarrow\overline{M}_{0,4}$ is the morphism forgetting the points labeled by $i_1,...,i_{n-4}$.\\
Furthermore, if $N = 1$ we have the evaluation maps 
\stepcounter{Theorem}
\begin{equation}\label{evaluation}
ev_i:\overline{M}_{0,n}(\mathbb{P}^1,d)\rightarrow\mathbb{P}^1
\end{equation}
Finally, if $N= d=1$ we have also the forgetful morphisms
$$\pi_{i_1,...,i_{n-1}}:\mathbb{P}^1[n]\cong\overline{M}_{0,n}(\mathbb{P}^1,1)\rightarrow\mathbb{P}^1[1]\cong \overline{M}_{0,1}(\mathbb{P}^1,1)\cong\mathbb{P}^1$$
forgetting the points labeled with $i_1,...,i_{n-1}$.

\begin{Remark}\label{forgeval}
The map $ev_1:\overline{M}_{0,1}(\mathbb{P}^1,1)\rightarrow\mathbb{P}^1$ is an isomorphism, and the diagram 
  \[
  \begin{tikzpicture}[xscale=3.5,yscale=-2.2]
    \node (A0_0) at (0, 0) {$\mathbb{P}^1[n]$};
    \node (A1_0) at (0, 1) {$\overline{M}_{0,1}(\mathbb{P}^1,1)\cong\mathbb{P}^1[1]$};
    \node (A1_1) at (1, 1) {$\mathbb{P}^1$};
    \path (A0_0) edge [->,swap]node [auto] {$\scriptstyle{\pi_{i_1,...,i_{n-1}}}$} (A1_0);
    \path (A1_0) edge [->]node [auto] {$\scriptstyle{ev_1}$} (A1_1);
    \path (A0_0) edge [->]node [auto] {$\scriptstyle{ev_{j}}$} (A1_1);
  \end{tikzpicture}
  \]
is commutative, where $\{j\} = \{1,...,n\}\setminus\{i_1,...,i_{n-1}\}$. Therefore, for any $j\in\{i,...,n\}$ the morphisms $ev_j$ and $\pi_{\{1,...,n\}\setminus\{j\}}$ are induced by the same base point free pencil on $\mathbb{P}^1[n]$.
\end{Remark}

\begin{Definition}\label{modmor}
A \textit{modular base point free pencil} on $\overline{M}_{0,n}(\mathbb{P}^1,d)$ is a linear system associated either to a forgetful morphism $\pi_{i_1,...,i_{n-4}}\circ\rho$ of type (\ref{forg}) or to an evaluation morphism $ev_i$ of type (\ref{evaluation}). 
\end{Definition}

Now let us consider the Fulton-MacPherson compactification $\mathbb{P}^1[n]$. By Proposition \ref{sym} we have that $\mathbb{P}^1[1]\cong\mathbb{P}^1$, $\mathbb{P}^1[2]\cong\mathbb{P}^1\times \mathbb{P}^1$, and $\mathbb{P}^1[3]\cong Bl_{\Delta_{1,2,3}}(\mathbb{P}^1)^{3}$, where $\Delta_{1,2,3}\subset (\mathbb{P}^1)^{3}$ is the small diagonal. This variety appears among the Fano $3$-folds of Picard number four in \cite{MM}. For the convenience of the reader we give a proof of the following well-known fact.

\begin{Lemma}\label{lemmaFano1}
The blow-up of $(\mathbb{P}^1)^{3}$ along the small diagonal $\Delta_{1,2,3}\subset (\mathbb{P}^1)^{3}$ is isomorphic to the blow-up of $\mathbb{P}^3$ along three skew lines $L_1,L_2,L_3$.
\end{Lemma} 
\begin{proof}
Let $\pi_i:\mathbb{P}^3\dasharrow\mathbb{P}^1$ be the projection with center the line $L_i$, and let us consider the rational map 
$$
\begin{array}{cccc}
\pi:&\mathbb{P}^3& \dasharrow & \mathbb{P}^1\times\mathbb{P}^1\times\mathbb{P}^1\\
 & x & \longmapsto & (\pi_1(x),\pi_2(x),\pi_3(x))
\end{array}
$$
The locus contracted by $\pi$ is the union of the lines in $\mathbb{P}^3$ intersecting the $L_i$'s, that is the unique quadric surface $Q$ containing $L_1,L_2,L_3$. Clearly $\pi(Q) = \Delta_{1,2,3}$.\\
Therefore, $\pi$ induces a birational morphism $\overline{\pi}:Bl_{L_1,L_2,L_3}\mathbb{P}^3\rightarrow(\mathbb{P}^1)^3$, whose exceptional locus is the strict transform $\widetilde{Q}$ of $Q$, and such that $\overline{\pi}(\widetilde{Q}) = \Delta_{1,2,3}$. Now, the universal property of the blow-up \cite[Proposition 7.14]{Har} yields a birational morphism $\widetilde{\pi}:Bl_{L_1,L_2,L_3}\mathbb{P}^3\rightarrow Bl_{\Delta_{1,2,3}}(\mathbb{P}^1)^3$ mapping $\widetilde{Q}$ to the exceptional divisor over $\Delta_{1,2,3}$. Finally, since $Bl_{L_1,L_2,L_3}\mathbb{P}^3$ and $Bl_{\Delta_{1,2,3}}(\mathbb{P}^1)^3$ are smooth and have the same Picard number $\widetilde{\pi}$ is an isomorphism.
\end{proof}

Let us denote by $E_i\subset\mathbb{P}^{1}[3]$ the exceptional divisor over $L_i$, and by $\widetilde{H}$ the pull-back of a general hyperplane of $\mathbb{P}^3$ via the blow-up morphism. Then $E_i\cong\mathbb{P}^1\times\mathbb{P}^1$. We will denote by $R_i,\sigma_i$ the classes of the two rulings of $E_i$, where $R_i$ is contracted by the blow-up map. Finally, let $\widetilde{L}$ be the pull-back of a general line in $\mathbb{P}^3$.\\
By \cite[Theorem 4.1]{DPU} the effective cone $\Eff(\mathbb{P}^1[3])$ of $\mathbb{P}^1[3]$ is the polyhedral cone generated by the extremal rays $2\widetilde{H}-E_1-E_2-E_3$, $H-E_i$ and $E_i$ for $i = 1,2,3$. Our aim is to describe its Mori Cone. 

\begin{Lemma}\label{MoriCone}
The Mori cone $\NE(\mathbb{P}^1[3])$ of $\mathbb{P}^1[3]$ is the polyhedral cone generated by $\widetilde{L}-R_1-R_2-R_3$, $R_i$ and $\sigma_i$ for $i=1,2,3$.
\end{Lemma}
\begin{proof}
Since $\mathbb{P}^1[3]$ is Fano we know that $\NE(\mathbb{P}^1[3])$ is a finitely generated polyhedral cone. Clearly, the classes $R_i,\sigma_i$ are extremal. Now, let $\widetilde{C}\subset\mathbb{P}^1[3]$ be an irreducible curve. If $\widetilde{C}\subset E_i$ we may write its class as a combination with non-negative coefficients of $R_i$ and $\sigma_i$. Therefore, let us assume that $\widetilde{C}\nsubseteq E_i$ for any $i=1,2,3$. In this case $\widetilde{C}$ is the strict transform of an irreducible curve $C\subset\mathbb{P}^3$ of degree $d$ and intersecting $L_i$ with multiplicity $m_i$. In other words we may write
$$\widetilde{C}\sim d\widetilde{L}-m_1R_1-m_2R_2-m_3R_3$$
where $\sim$ denotes numerical equivalence. Now, note that
$$\widetilde{C}\sim d(\widetilde{L}-R_1-R_2-R_3)+(d-m_1)R_1+(d-m_2)R_2+(d-m_3)R_3$$
To conclude it is enough to observe that $d-m_i\geq 0$ for any $i=1,2,3$, otherwise by Bezout's theorem the line $L_i$ would be an irreducible component of $C$.
\end{proof}

We will need the following observation.

\begin{Remark}\label{delP}
By \cite[Sections 6.1 and 6.4]{Has} the moduli space of weighted pointed rational curves $\overline{M}_{0,A[5]}$ with weights $A[5] = (1,1,1/3,1/3,1/3)$ can be constructed as the blow-up $Bl_{p_1,p_2,p_3}\mathbb{P}^2$ of $\mathbb{P}^2$ at three general points $p_1,p_2,p_3\in\mathbb{P}^2$. Such moduli space admits three morphisms $\pi_i:\overline{M}_{0,A[5]}\rightarrow\overline{M}_{0,4}\cong\mathbb{P}^1$ given by forgetting one of the last three marked points. These morphisms are the lifting of the projections $\mathbb{P}^2\dasharrow\mathbb{P}^1$ with center in one of the three blown-up points. Now, by \cite[Proposition 2.2]{MaM17} any dominant morphism $\overline{M}_{0,A[5]}\rightarrow\overline{M}_{0,4}\cong\mathbb{P}^1$ factors through one of the $\pi_i$'s. Note that \cite[Proposition 2.2]{MaM17} is stated for morphisms with connected fibers $\overline{M}_{0,A[5]}\rightarrow\overline{M}_{0,4}\cong\mathbb{P}^1$. However, to get the result for any dominant morphism $f:\overline{M}_{0,A[5]}\rightarrow\mathbb{P}^1$ it is enough to consider the Stein factorization $h:\overline{M}_{0,A[5]}\rightarrow C$ of $f:\overline{M}_{0,A[5]}\rightarrow\mathbb{P}^1$, and the normalization $\nu:\widetilde{C}\rightarrow C$ of the curve $C$, as in the proofs of \cite[Lemma 3.5 and Corollary 3.8]{BM}.

In other words any morphism $X\rightarrow\mathbb{P}^1$, where $X$ is a del Pezzo surface of degree six given by blowing-up three general points in $\mathbb{P}^2$, factors through one of the three morphisms induced by the linear projection from the blown-up points. 

I believe that this result on fibrations of degree six del Pezzo surfaces has been known for a long time but I could not find a classical reference in the literature.  
\end{Remark} 

The alternative description of $\mathbb{P}^1[3]$ in Lemma \ref{lemmaFano1} and Remark \ref{delP} are helpful in order to classify the base point free pencils on $\mathbb{P}^1[3]$.

\begin{Lemma}\label{crem}
Let $\psi:\mathbb{P}^1[3]\rightarrow\mathbb{P}^1$ be a dominant morphism. Then $f$ factors through an evaluation map.
\end{Lemma}
\begin{proof}
By Lemma \ref{lemmaFano1} we can identify $\mathbb{P}^1[3]$ with the blow-up of $\mathbb{P}^3$ along three skew lines $L_1,L_2,L_3$. The statement follows easily from the description of $\NE(\mathbb{P}^1[3])$ in Lemma \ref{MoriCone}. However, we will give an alternative easy and geometrical proof.\\
The morphism $\psi$ induces a rational map $\widetilde{\psi}:\mathbb{P}^3\dasharrow\mathbb{P}^1$ whose indeterminacy locus is contained in $L_1\cup L_2\cup L_3$. Let $H\subset\mathbb{P}^3$ be a general plane. Then the restriction $\widetilde{\psi}_{|H}:H\cong\mathbb{P}^2\dasharrow\mathbb{P}^1$ is a rational map whose indeterminacy locus is a finite set $S$ contained in $\{L_1\cap H, L_2\cap H, L_3\cap H\}$, and inducing a morphism $\psi_{|H}:Bl_{S}H\rightarrow\mathbb{P}^1$.\\
If $S = \{p_1,p_2,p_3\}$, that is $Bl_{S}H$ is a del Pezzo surface of degree six, then by Remark \ref{delP} $\widetilde{\psi}_{|H}$ must factor through a linear projection from one of the $p_i$'s. A fortiori the same result holds if $|S|\in\{1,2\}$.

Therefore, $\widetilde{\psi}_{|H}$ factors through the projection from one of the $p_i$'s, and hence $\widetilde{\psi}$ is constant on the fibers of the projection $\pi_{L_i}$ from one of the $L_i$'s. Then $\widetilde{\psi}$ factors through $\pi_{L_i}$. Via the isomorphism $\widetilde{\pi}:Bl_{L_1,L_2,L_3}\mathbb{P}^3\rightarrow Bl_{\Delta_{1,2,3}}(\mathbb{P}^1)^3$, this means that $\psi$ factors through the lifting of one of the three projections $(\mathbb{P}^1)^3\rightarrow\mathbb{P}^1$, that is an evaluation map.
\end{proof}

Now, we are ready to prove the main result of this section on base point free pencils on the Fulton-MacPherson compactification $\mathbb{P}^1[n]$.

\begin{Proposition}\label{propfact}
Let $\psi:\mathbb{P}^1[n]\rightarrow\mathbb{P}^1$ be a dominant morphism. Then $\psi$ factors through a morphism associated to a modular base point free pencil.
\end{Proposition}
\begin{proof}
By Proposition \ref{sym} $\mathbb{P}^1[n]$ is the blow-up $g_{n}:\mathbb{P}^1[n]\rightarrow (\mathbb{P}^1)^n$ along the diagonals in order of increasing dimension. Now, let $b_1:\mathbb{P}^1[n]_1\rightarrow (\mathbb{P}^1)^n$ be the blow-up of the smallest diagonal $\Delta_{1,...,n}$, and fixed a point $p\in \Delta_{1,...,n}$ let us consider the fiber $b_1^{-1}(p)\cong\mathbb{P}^{n-2}$. Finally, let $b:\mathbb{P}^1[n]\rightarrow\mathbb{P}^1[n]_1$ be the blow-up morphism such that $b_1\circ b = g_n$.

For a proper subset $S = \{i_1,...,i_s\}$ of $\{1,...,n\}$ we will denote by $\Delta_S$ the diagonal in $(\mathbb{P}^1)^n$ where the points indexed by $S$ coincide. Let $E$ be the exceptional divisor of the first blow-up $\mathbb{P}^1[n]_1$ of $(\mathbb{P}^1)^n$ at the smallest diagonal $\Delta_{1,...,n}$, and let $\widetilde{E}$ be the boundary divisor in $\mathbb{P}^1[n]$ corresponding to $\Delta_{1,...,n}$. Now, let $F_S$ be the intersections of $E$ with the strict transforms of $\Delta_S$, for $2 \leq |S|\leq n-1$ inside $\mathbb{P}^1[n]_1$. By the symmetric construction of the Fulton-MacPherson compactification in Proposition \ref{sym} we have that $\widetilde{E}$ is then the iterated blow-up of $E$ at $F_S$ in order of increasing dimension. Now, by \cite[Lemma 4.18]{GR17} we get that the fiber of $\widetilde{E}\rightarrow\Delta_{1,...,n}$ over $p\in \Delta_{1,...,n}$ coincides
with the iterated blow-up of the fiber $E_p$ at $F_{S,p}$, where $E_p$ is the fiber of $E$ over $p\in \Delta_{1,...,n}$ and $F_{S,p}=F_S\cap E_p$, in order of increasing dimension. Indeed in the notations of \cite[Lemma 4.18]{GR17}, setting $d = 1$ and all weights in $\mathcal{A}$ equal to one, we have that $b_{|g_{n}^{-1}(p)}:g_{n}^{-1}(p)\rightarrow b_1^{-1}(p)\cong\mathbb{P}^{n-2}$ is exactly the blow-up morphism in the construction of $\overline{M}_{0,n+1}$ in \cite[Section 6.2]{Has}.

Note that the image of the morphism $f_p:\overline{M}_{0,n+1}\rightarrow\mathbb{P}^1[n]$ in Section \ref{cones} is the fiber $g_{n}^{-1}(p)\cong\overline{M}_{0,n+1}$. 
Let us assume that $\psi$ contracts $g_{n}^{-1}(p)$ for some $p\in\Delta_{1,...,n}$. Therefore, $\psi$ contracts $g_{n}^{-1}(p)$ for any $p\in\Delta_{1,...,n}$ since the fibers $g_{n}^{-1}(p)$ are all numerically equivalent. Now, let $E_S = g_{n}^{-1}(\Delta_S)$ be the exceptional divisor over $\Delta_S$, and $g_{n,S}:E_S\rightarrow\Delta_S$ be the restriction of $g_{n}$ to $E_S$. For any $p\in \Delta_{1,...,n}$ we have $g_{n,S}^{-1}(p)\subset g_{n}^{-1}(p)$. Therefore, $\psi$ contracts the fiber $g_{n,S}^{-1}(p)$.\\
Now, since $g_{n,S}:E_S\rightarrow\Delta_S$ is a morphism with connected fibers all of the same dimension, \cite[Lemma 1.6]{KM} yields that $\psi$ must contract all the fibers of $g_{n,S}:E_S\rightarrow\Delta_S$. We conclude that if $\psi$ contracts a fiber $g_{n}^{-1}(p)\cong\overline{M}_{0,n+1}$ then it must contract all the fibers of any morphism $g_{n,S}:E_S\rightarrow\Delta_S$ with $|S| = n-r$, $r \in\{1,...,n-3\}$, and hence $\psi$ factors through the blow-up morphism $g_{n}:\mathbb{P}^1[n]\rightarrow (\mathbb{P}^1)^n$, that is there exists a commutative diagram as follows
  \[
  \begin{tikzpicture}[xscale=2.1,yscale=-1.7]
    \node (A0_0) at (0, 0) {$\mathbb{P}^1[n]$};
    \node (A1_0) at (0, 1) {$(\mathbb{P}^1)^n$};
    \node (A1_1) at (1, 1) {$\mathbb{P}^1$};
    \path (A0_0) edge [->,swap]node [auto] {$\scriptstyle{g_{n}}$} (A1_0);
    \path (A1_0) edge [->]node [auto] {$\scriptstyle{}$} (A1_1);
    \path (A0_0) edge [->]node [auto] {$\scriptstyle{\psi}$} (A1_1);
  \end{tikzpicture}
  \]
Since by Lemma \ref{fibg2} any morphism $(\mathbb{P}^1)^n\rightarrow\mathbb{P}^1$ factors through a projection onto one of the factors we get that $\psi$ factors through an evaluation map.\\
Now, let us assume that $\psi$ restricts to a dominant morphism on $g_{n}^{-1}(p)\cong\overline{M}_{0,n+1}$. Following Notation \ref{not} let $D^{\psi}\equiv D_{0,n+1}^{\psi}+a_1^{\psi}H_1+...+a_n^{\psi}H_n$ be the decomposition of the divisor $D^{\psi}$ associated to the morphism $\psi:\mathbb{P}^1[n]\rightarrow\mathbb{P}^1$.\\
Note that by Section \ref{cones} the divisor $H_i$ induces the evaluation map $ev_i$. In particular, since any $ev_i$ contracts the fiber $g_{n}^{-1}(p)$ we have that $a_i^{\psi} = 0$ for any $i=1,...,n$, and $D^{\psi}\equiv D_{0,n+1}^{\psi}$.\\
By \cite[Corollary 3.8]{BM} $\psi_{|\overline{M}_{0,n+1}}$ factors through a forgetful morphism $\pi_{i_1,...,i_{n-3}}$ and a finite morphism. Now, we distinguish two cases.
\begin{itemize}
\item[-] Assume that $n+1\in\{i_1,...,i_{n-3}\}$, say $n+1 = i_{n-3}$. In this case the morphism  
$$\mathbb{P}^1[n]\cong\overline{M}_{0,n}(\mathbb{P}^1,1)\xrightarrow{\makebox[1.5cm]{$\rho$}}\overline{M}_{0,n}\xrightarrow{\pi_{i_1,...,i_{n-4}}}\overline{M}_{0,4}\cong\mathbb{P}^1$$
restricts to $\psi$ on $\overline{M}_{0,n+1}$. Furthermore, since (\ref{maph}) is an isomorphism it is the only morphism with this property, and hence $\psi$ factors though $\pi_{i_1,...,i_{n-4}}\circ\rho$.
\item[-] Now, assume that $n+1\notin\{i_1,...,i_{n-3}\}$. In this case the forgetful morphism
$$\pi_{i_1,...,i_{n-3}}:\mathbb{P}^1[n]\cong\overline{M}_{0,n}(\mathbb{P}^1,1)\rightarrow\mathbb{P}^1[3]\cong\overline{M}_{0,3}(\mathbb{P}^1,1)$$
coincides with $\psi$ on $\overline{M}_{0,n+1}$, and by (\ref{maph}) $\psi$ must factor through $\pi_{i_1,...,i_{n-3}}$, that is $\psi = \xi\circ\pi_{i_1,...,i_{n-3}}$ where $\xi:\mathbb{P}^1[3]\rightarrow\mathbb{P}^1$ is a morphism. On the other hand, by Lemma \ref{crem} the morphism $\xi$ factors through an evaluation morphism $ev_i:\mathbb{P}^1[3]\rightarrow\mathbb{P}^1$. Note that $(ev_i\circ\pi_{i_1,...,i_{n-3}}\circ \psi)(\overline{M}_{0,n+1})$ is a point for any $i$, and hence $(\xi\circ\pi_{i_1,...,i_{n-3}}\circ f)(\overline{M}_{0,n+1})$ is a point as well. A contradiction.  
\end{itemize} 
We conclude that $\psi$ factors either through an evaluation map or a morphism of the type $\pi_{i_1,...,i_{n-4}}\circ\rho$ as in (\ref{forg}).
\end{proof}

\section{On the fibrations of $X[n]$}\label{sectionFib}

Our next aim is to describe dominant morphisms $\mathbb{P}^1[n]\rightarrow\mathbb{P}^{1}[r]$ with $r\geq 2$. The case $r = 2$ is an immediate consequence of Proposition \ref{propfact}.

\begin{Proposition}\label{r=2}
Let $\psi:\mathbb{P}^1[n]\rightarrow\mathbb{P}^1[2]$ be a dominant morphism. Then $\psi$ factors through a product of two morphisms associated to modular base point free pencils.
\end{Proposition}
\begin{proof}
Since $\mathbb{P}^1[2]\cong\mathbb{P}^1\times\mathbb{P}^1$, the morphism $\psi$ is completely determined by the two morphisms $\pi_i\circ\psi:\mathbb{P}^1[n]\rightarrow\mathbb{P}^1$, where the $\pi_i:\mathbb{P}^1[2]\rightarrow\mathbb{P}^1$ for $i = 1,2$ are the projections onto the factors. Now, the statement follows immediately from Proposition \ref{propfact}.
\end{proof}

When $r\geq 3$ the geometry of $\mathbb{P}^1[r]$ radically changes. This is because now inside $\mathbb{P}^{1}[r]$ there are negative divisors imposing several constraints on morphisms $\mathbb{P}^1[n]\rightarrow\mathbb{P}^{1}[r]$. Indeed, the case $r = 3$ is the first one in which we really need to blow-up a codimension two subvariety of $(\mathbb{P}^1)^r$ to construct $\mathbb{P}^{1}[r]$. This will be the leading idea for the rest of this section.

\begin{Lemma}\label{lemmaftot1}
Let $\psi:\mathbb{P}^1[n]\rightarrow\mathbb{P}^1[3]$ be a dominant morphism with connected fibers. Then for any $i\in \{1,2,3\}$ the morphism $ev_i\circ\psi:\mathbb{P}^1[n]\rightarrow\mathbb{P}^1$ factors through an evaluation morphism $ev_{j_i}:\mathbb{P}^1[n]\rightarrow\mathbb{P}^1$.
\end{Lemma}
\begin{proof}
The morphism $ev_i\circ\psi:\mathbb{P}^1[n]\rightarrow\mathbb{P}^1$ has connected fibers. Therefore, by Proposition \ref{propfact} we know that $ev_i\circ\psi$ factors either through an evaluation morphism $ev_{j_i}$ and an automorphism $\mu_i\in PGL(2)$, or through a morphism of the type $\pi_{i_1,...,i_{n-4}}\circ\rho$ in (\ref{forg}) and again an automorphism $\mu_i\in PGL(2)$. Let $\xi_i:\mathbb{P}^1[n]\rightarrow\mathbb{P}^1$ be the morphism, either of type (\ref{forg}) or of type (\ref{evaluation}), factorizing $ev_i\circ\psi$. Note that $\mu:=\mu_1\times\mu_2\times\mu_3$ is an automorphism of $(\mathbb{P}^1)^3$, and we have the following commutative diagram
\[
\begin{tikzpicture}[xscale=2.5,yscale=-1.2]
    \node (A0_0) at (0, 0) {$\mathbb{P}^1[n]$};
    \node (A0_1) at (1, 0) {$\mathbb{P}^1[3]$};
    \node (A1_0) at (0, 1) {$(\mathbb{P}^1)^3$};
    \node (A1_1) at (1, 1) {$(\mathbb{P}^1)^3$};
    \path (A0_0) edge [->]node [auto] {$\scriptstyle{\psi}$} (A0_1);
    \path (A1_0) edge [->]node [auto] {$\scriptstyle{\mu}$} (A1_1);
    \path (A0_1) edge [->]node [auto] {$\scriptstyle{ev_1\times ev_2\times ev_3}$} (A1_1);
    \path (A0_0) edge [->]node [auto,swap] {$\scriptstyle{\xi_1\times \xi_2\times \xi_3}$} (A1_0);
  \end{tikzpicture} 
\]
Recall that by Remark \ref{kfm} the morphism $ev_1\times ev_2\times ev_3$ is nothing but the blow-up morphism $\mathbb{P}^1[3]\rightarrow(\mathbb{P}^1)^3$, that is the blow-up of the diagonal $\Delta_{1,2,3}$.\\ 
Now, let $x\in \Delta_{1,2,3}\subset (\mathbb{P}^1)^3$ be a point, and let $F_x$ be the fiber of $ev_1\times ev_2\times ev_3$ over $x$. Then $\dim(F_x) = 1$, and if $\overline{F}_x$ is a component of $\psi^{-1}(F_x)$ we have $\dim(\overline{F}_x)\geq n-3+1 = n-2$. On the other hand, $\overline{F}_x$ is contracted to the point $y = \mu^{-1}(x)$ by $\xi_1\times\xi_2\times\xi_3$. Now, since $\xi_i$ is either of type (\ref{forg}) or of type (\ref{evaluation}), $\overline{F}_x\subseteq (\xi_1\times\xi_2\times\xi_3)^{-1}(y)$ and $\dim(\overline{F}_x)\geq n-2$ yield $y\in\Delta_{1,2,3}$.\\ 
Hence $\mu\in \Aut((\mathbb{P}^1)^3)$ preserves $\Delta_{1,2,3}$. For instance, let us assume $\xi_1 = ev_1$, $\xi_2 = ev_2$ and $\xi_3 = \pi_I\circ\rho$, and let $(p,q,t)\in(\mathbb{P}^1)^3$ be a point. The other cases can be worked out with similar arguments. If $p\neq q$ then a general point in $(\xi_1\times\xi_2\times\xi_3)(p,q,t)$ is a stable map of the form $[C,(\alpha^{-1}(p),\alpha^{-1}(q),x_3,...,x_n),\alpha]$ such that $\xi_3([C,(\alpha^{-1}(p),\alpha^{-1}(q),x_3,...,x_n),\alpha]) =t$. Therefore, for any set $I$ of $n-4$ indices, in the notation above, we have $\dim(\overline{F}_x)\geq n-3$. This yields $p = q$, and hence a general point in $(\xi_1\times\xi_2\times\xi_3)(p,q,t)$ is a stable map of the form $[C_1\cup C_2,(x_1,...,x_n),\alpha]$ where $x_1,x_2\in C_1$, $\alpha$ has degree zero on $C_1$ and one on $C_2$, such that $\xi_3([C_1\cup C_2,(x_1,...,x_n),\alpha]) =t$.\\ 
Now, $\dim(\overline{F}_x) = n-2$ if and only if $\{1,...,n\}\setminus I = \{1,2,i_3,i_4\}$ with $x_{i_3},x_{i_4}\in C_2$. Furthermore, in this case $\xi_3([C_1\cup C_2,(x_1,...,x_n),\alpha])\in\overline{M}_{0,4}$ is the stable curve $[C_1\cup C_2,(x_1,x_2,x_{i_3},x_{i_4})]$  which corresponds to the point of $\overline{M}_{0,4}\cong\mathbb{P}^1$ representing a $4$-pointed stable curve where $x_1$ and $x_2$ collide. Therefore $(p,q,t)\in \Delta_{1,2,3}$.\\
Now, let $E\subset\mathbb{P}^1[3]$ be the exceptional divisor on $\Delta_{1,2,3}$. Since $\mathbb{P}^1[3]$ is smooth any component of $D = \phi^{-1}(E)\subset\mathbb{P}^{1}[n]$ has codimension one. Furthermore $((\mu_1\circ\xi_1)\times(\mu_2\circ\xi_2)\times(\mu_3\circ\xi_3))(D) = \Delta_{1,2,3}$. In particular, since $\mu^{-1}(\Delta_{1,2,3}) = \Delta_{1,2,3}$, any component of $(\xi_1\times\xi_2\times\xi_3)^{-1}(\Delta_{1,2,3})$ must have codimension one in $\mathbb{P}^1[n]$.\\
For instance, if the $\xi_i$'s are all evaluation maps, say $ev_1, ev_2, ev_3$ then the irreducible components of $(\xi_1\times\xi_2\times\xi_3)^{-1}(\Delta_{1,2,3})$ are the divisors $D_{1,2,3,i_1,...,i_k}$ with $0\leq k\leq n-3$, where a general point of $D_{1,2,3,i_1,...,i_k}$ corresponds to a stable map $[C_1\cup C_2,(x_1,...,x_n),\alpha]$ with $x_1,x_2,x_3,x_{i_2},...,x_{i_k}\in C_1$, the remaining marked points are in $C_2$, and $\alpha$ has degree zero and one on $C_1$ and $C_2$ respectively.\\
Now let us assume that $\xi_1$ is a morphism of type (\ref{forg}), say $\pi_{i_1,...,i_{n-4}}\circ\rho$, while $\xi_2, \xi_3$ are the evaluation maps $ev_{1}, ev_{2}$. A general point in the exceptional divisor $E_{1,...,n}\subset\mathbb{P}^1[n]$ over the diagonal $\Delta_{1,...,n}\subset(\mathbb{P}^{1})^{n}$ represents a stable map $[C_1\cup C_2,(x_1,...,x_n),\alpha]$, where $x_{1},...,x_{n}\in C_1$, and $\alpha$ has degree zero and one on $C_1$ and $C_2$ respectively. Clearly $ev_1([C_1\cup C_2,(x_1,...,x_n),\alpha]) = ev_2([C_1\cup C_2,(x_1,...,x_n),\alpha])$. Now, the condition $(\pi_{i_1,...,i_{n-4}}\circ\rho)([C_1\cup C_2,(x_1,...,x_n),\alpha]) = [C_2,(x_{i_1},...,x_{i_4})] = \alpha(x_1) = \alpha(x_2)$ cuts out a codimension two component $L_1\subset E_{1,...,n}$ of  $((\pi_{i_1,...,i_{n-4}}\circ\rho)\times ev_1\times ev_2)^{-1}(\Delta_{1,2,3})$. A contradiction.\\
Now, let us consider the case $\xi_1 = \pi_{i_1,...,i_{n-4}}\circ\rho$, $\xi_2 = \pi_{j_1,...,j_{n-4}}\circ\rho$, and $\xi_3 = ev_1$. Consider the divisor whose general point corresponds to a stable map $[C_1\cup C_2,(x_1,...,x_n),\alpha]$, where $x_1,x_{i_1},x_{i_2},x_{j_1},x_{j_2}\in C_1$, $x_{i_3},x_{i_4},x_{j_3},x_{j_4}\in C_2$, and $\alpha$ has degree one on $C_1$ and zero on $C_2$. Clearly, $(\pi_{i_1,...,i_{n-4}}\circ\rho)([C_1\cup C_2,(x_1,...,x_n),\alpha]) = (\pi_{j_1,...,j_{n-4}}\circ\rho)([C_1\cup C_2,(x_1,...,x_n),\alpha]) = [C_1\cap C_2, (x_{i_1},...,x_{i_4})]$. On the other hand, since $x_1\in C_1$ and $\alpha$ does not contract $C_1$ the condition $ev_1([C_1\cup C_2,(x_1,...,x_n),\alpha]) = [C_1\cup C_2, (x_{i_1},...,x_{i_4})]$ defines a codimension two locus $L_2\subset\mathbb{P}^1[n]$ which is a component of $((\pi_{i_1,...,i_{n-4}}\circ\rho)\times (\pi_{j_1,...,j_{n-4}}\circ\rho)\times ev_2)^{-1}(\Delta_{1,2,3})$. A contradiction.\\
Finally, we consider the case $\xi_1 = \pi_{i_1,...,i_{n-4}}\circ\rho$, $\xi_2 = \pi_{j_1,...,j_{n-4}}\circ\rho$, and $\xi_3 = \pi_{k_1,...,k_{n-4}}\circ\rho$. The sets $I=\{i_1,...,i_{4}\}$, $J=\{j_1,...,j_{4}\}$, $K=\{k_1,...,k_{4}\}$ differ by at least one element, say $i_4\notin J\cup K$, $j_4\notin I\cup K$, $k_4\notin I\cup J$. In this case the codimension two locus $L_3\subset\mathbb{P}^1[n]$, whose general point corresponds to a curve $[C_1\cup C_2\cup C_3,(x_1,...,x_n),\alpha]$ where $
x_{i_4}\in C_2$, $x_{j_4},x_{k_4}\in C_3$ and the remaining marked points are in $C_1$, is a component of $((\pi_{i_1,...,i_{n-4}}\circ\rho)\times (\pi_{j_1,...,j_{n-4}}\circ\rho)\times (\pi_{k_1,...,k_{n-4}}\circ\rho))^{-1}(\Delta_{1,2,3})$ of codimension two. Again a contradiction. 
\end{proof}

In order to prove our main result, given two forgetful morphisms $\pi_I, \pi_J$,  we need to control the dimension of the intersection of two general fibers of $\pi_I$ and $\pi_J$.

\begin{Lemma}\label{intfib}
Fix a point $x\in\mathbb{P}^1[n]$ not lying in the exceptional locus of the blow-up morphism $\mathbb{P}^1[n]\rightarrow (\mathbb{P}^1)^n$. Let $\pi_I:\mathbb{P}^1[n]\rightarrow\mathbb{P}^{1}[r-1]$, $\pi_J:\mathbb{P}^1[n]\rightarrow\mathbb{P}^{1}[r-1]$ be two forgetful morphisms, and $F_{I,x}$, $F_{J,x}$ be the fibers through $x$ of $\pi_{I}$ and $\pi_{J}$ respectively. If $\dim(F_{I,x}\cap F_{J,x})\geq n-r$ then $|I\cap J|\geq n-r$.
\end{Lemma}
\begin{proof}
Let $|I\cap J| = n-r-k$. We may assume $I = \{1,...,n-r-k,n-r-k+1,...,n-r+1\}$ and $J = \{1,...,n-r-k,n-r+2,...,n-r+k+2\}$.\\
A point $y\in F_{I,x}\cap F_{J,x}$ represents a pointed stable map where all the marked points but $x_{1},...,x_{n-r-k}$ are fixed, that is $F_{I,x}\cap F_{J,x}$ parametrizes configurations of $n-r-k$ points in $\mathbb{P}^1$, and $\dim(F_{I,x}\cap F_{J,x}) = n-r-k$. Now, $\dim(F_{I,x}\cap F_{J,x}) = n-r-k\geq n-r$ yields $k \leq 0$.
\end{proof}

Now, we are ready to prove the main result of this section on the classification of morphisms $\mathbb{P}^1[n]\rightarrow\mathbb{P}^{1}[r]$.

\begin{Theorem}\label{thfact}
Let $\psi:\mathbb{P}^1[n]\rightarrow\mathbb{P}^{1}[r]$ be a dominant morphism with connected fibers. If $r\geq 3$ then $\psi$ factors through a forgetful morphism $\pi_I:\mathbb{P}^1[n]\rightarrow\mathbb{P}^1[r]$.
\end{Theorem}
\begin{proof}
Let us begin with the case $r = 3$. Keeping in mind that by Remark \ref{forgeval} we may identify evaluation maps $ev_i:\mathbb{P}^{1}[n]\rightarrow\mathbb{P}^{1}$ with forgetful morphisms forgetting $n-1$ marked points, by Lemma \ref{lemmaftot1} we get the following commutative diagram  
 \[
  \begin{tikzpicture}[xscale=4.5,yscale=-1.2]
    \node (A0_0) at (0, 0) {$\mathbb{P}^1[n]$};
    \node (A0_1) at (1, 0) {$\mathbb{P}^1[3]$};
    \node (A1_1) at (1, 1) {$(\mathbb{P}^{1})^3$};
    \path (A0_0) edge [->]node [auto] {$\scriptstyle{\psi}$} (A0_1);
    \path (A0_1) edge [->]node [auto] {$\scriptstyle{ev_{1}\times ev_{2}\times ev_3}$} (A1_1);
    \path (A0_0) edge [->]node [auto,swap] {$\scriptstyle{\mu\circ(\pi_{I}\times\pi_{J}\times \pi_{K})}$} (A1_1);
  \end{tikzpicture}
  \]
where $\mu$ is an automorphism of $(\mathbb{P}^1)^3$ preserving the diagonal $\Delta_{1,2,3}$, and we may assume $I = \{2,...,n\}$, $J = \{1,3,...,n\}$, $K = \{1,2,4,...,n\}$. Since $\mu$ preserves the diagonal it lifts to an automorphism $\overline{\mu}:\mathbb{P}^{1}[3]\rightarrow\mathbb{P}^{1}[3]$. Therefore, the morphism $\overline{\mu}\circ\pi_{4,...,n}:\mathbb{P}^1[n]\rightarrow\mathbb{P}^{1}[3]$ and the morphism $\psi$ coincide on a dense open subset of $\mathbb{P}^{1}[n]$, and hence they are the same morphism. This means that $\psi = \overline{\mu}\circ\pi_{4,...,n}$.\\
Now, we proceed by induction on $r\geq 4$. Let us consider two different forgetful morphisms $\pi_i,\pi_j:\mathbb{P}^{1}[r]\rightarrow\mathbb{P}^1[r-1]$. Since $r-1\geq 3$ by induction hypothesis we have the following diagram
\[
  \begin{tikzpicture}[xscale=3.5,yscale=-1.2]
    \node (A0_1) at (1, 0) {$\mathbb{P}^{1}[r-1]$};
    \node (A0_2) at (2, 0) {$\mathbb{P}^{1}[r-1]$};
    \node (A1_0) at (0, 1) {$\mathbb{P}^{1}[n]$};
    \node (A1_1) at (1, 1) {$\mathbb{P}^{1}[r]$};
    \node (A2_1) at (1, 2) {$\mathbb{P}^{1}[r-1]$};
    \node (A2_2) at (2, 2) {$\mathbb{P}^{1}[r-1]$};
    \path (A1_0) edge [->,swap]node [auto] {$\scriptstyle{\pi_J}$} (A2_1);
    \path (A2_1) edge [->]node [auto] {$\scriptstyle{}$} (A2_2);
    \path (A1_1) edge [->,swap]node [auto] {$\scriptstyle{\pi_i}$} (A0_2);
    \path (A1_0) edge [->]node [auto] {$\scriptstyle{\pi_I}$} (A0_1);
    \path (A1_0) edge [->]node [auto] {$\scriptstyle{\psi}$} (A1_1);
    \path (A1_1) edge [->]node [auto] {$\scriptstyle{\pi_j}$} (A2_2);
    \path (A0_1) edge [->]node [auto] {$\scriptstyle{}$} (A0_2);
  \end{tikzpicture}
\]
In the notation of Lemma \ref{intfib} let $x\in\mathbb{P}^{1}[n]$ be a general point, and $F_{I,x}$, $F_{J,x}$ be the fibers of $\pi_I$, $\pi_J$ through $x$. Note that the fiber $F_{\psi,x}$ of $\psi$ through $x$ is contained in the intersection $F_{I,x}\cap F_{J,x}$. Therefore, Lemma \ref{intfib} yields that $\pi_I$ and $\pi_J$ forget $n-r$ common marked points. If $\pi_{I\cap J}:\mathbb{P}^1[n]\rightarrow\mathbb{P}^1[r]$ is the morphism forgetting these $n-r$ common points we have
$$F_{\psi,x}\subseteq F_{I,x}\cap F_{J,x} = F_{I\cap J,x}$$
Since $\dim(F_{\psi,x}) = n-r = \dim(F_{I\cap J,x})$ we get $F_{\psi,x} = F_{I\cap J,x}$. This means that $\psi$ contracts to a point the general fiber of $\pi_{I\cap J}$. Finally, since $\pi_{I\cap J}$ is a morphism with connected fibers all of the same dimension, \cite[Lemma 1.6]{KM} yields that $\psi$ contracts to a point all the fibers of $\pi_{I\cap J}$. This means that given a point $y\in\mathbb{P}^{1}[r]$ we have that $z = \psi(\pi_{I\cap J}^{-1}(y))$ is a point as well, and we get the commutative diagram
  \[
  \begin{tikzpicture}[xscale=2.2,yscale=-1.4]
    \node (A0_0) at (0, 0) {$\mathbb{P}^1[n]$};
    \node (A1_0) at (0, 1) {$\mathbb{P}^1[r]$};
    \node (A1_1) at (1, 1) {$\mathbb{P}^1[r]$};
    \path (A0_0) edge [->,swap]node [auto] {$\scriptstyle{\pi_{I\cap J}}$} (A1_0);
    \path (A1_0) edge [->]node [auto] {$\scriptstyle{\nu}$} (A1_1);
    \path (A0_0) edge [->]node [auto] {$\scriptstyle{\psi}$} (A1_1);
  \end{tikzpicture}
  \]
where $\nu:\mathbb{P}^{1}[r]\rightarrow\mathbb{P}^{1}[r]$ is defined by $\nu(y) = z$.
\end{proof}

A small improvement is at hand. 

\begin{Corollary}\label{corfactprod}
Let $\psi:\mathbb{P}^1[n]\rightarrow\mathbb{P}^{1}[r_1]\times ...\times\mathbb{P}^{1}[r_k]$ be a dominant morphism with connected fibers. Then $\psi$ factors through a product of forgetful morphisms of type $\pi_I$ (\ref{forg1}) and $\rho\circ\pi_J$ (\ref{forg}). Furthermore, if $r_i\geq 3$ for any $i = 1,...,k$ then $\psi$ factors through a product of forgetful morphisms of type $\pi_I$ (\ref{forg1}) only.
\end{Corollary}
\begin{proof}
It is enough to compose $\psi$ with projections onto the factors, and to apply Propositions \ref{propfact}, \ref{r=2}, and Theorem \ref{thfact}.
\end{proof}

Finally, we extend the main results in Sections \ref{LinPenc} and \ref{sectionFib} for varieties not containing rational curves.

\begin{Lemma}\label{noratcurves}
Let $X$ be a smooth projective variety not containing any rational curve, and $\psi:X[n]\rightarrow X^r$ be a dominant morphism. Then $\psi$ factors through the blow-up morphism $g_n:X[n]\rightarrow X^n$.  
\end{Lemma}
\begin{proof}
It is enough to prove the claim in the case $r = 1$. If $r\geq 2$ we just consider the composition of $\psi$ with the projections $X^r\rightarrow X$.\\
So, let $\psi:X[n]\rightarrow X$ be a dominant morphism, $x\in X^n$ a point, and $F_x = g_n^{-1}(x)$. Then $F_x$ is a rational variety. Assume that $F_x$ has positive dimension and that $\psi$ does not contract $F_x$ to a point in $X$. Then $\psi(F_x)\subseteq X$ is a rationally connected variety of positive dimension. A contradiction, since by hypothesis $X$ does not contain rational curves. Therefore, $\psi$ must contract any fiber of $g_n$. 
\end{proof}

It is straightforward to prove an analogue of Corollary \ref{corfactprod} for $C[n]$ where $C$ is a smooth projective curve with $g(C)\geq 2$. We will denote by $\pi_I:C[n]\rightarrow C[r]$ the forgetful morphisms. These are just the liftings of the projections $C^n\rightarrow C^r$.

\begin{Proposition}\label{propfactg2}
Let $C$ be a smooth projective curve of genus $g(C)\geq 2$, and $\psi:C[n]\rightarrow C[r_1]\times ...\times C[r_k]$ be a dominant morphism with connected fibers. Then $\psi$ factors through a product of forgetful morphisms. 
\end{Proposition}
\begin{proof}
First let us consider a morphism $\psi:C[n]\rightarrow C[r]$. Let $g_r:C[r]\rightarrow C^r$ be the blow-up, and let us consider the composition $g_r\circ \psi:C[n]\rightarrow C^r$. By Lemma \ref{noratcurves} we have the following commutative diagram:
  \[
  \begin{tikzpicture}[xscale=2.5,yscale=-1.2]
    \node (A0_0) at (0, 0) {$C[n]$};
    \node (A0_1) at (1, 0) {$C[r]$};
    \node (A1_0) at (0, 1) {$C^n$};
    \node (A1_1) at (1, 1) {$C^r$};
    \path (A0_0) edge [->]node [auto] {$\scriptstyle{\psi}$} (A0_1);
    \path (A0_0) edge [->,swap]node [auto] {$\scriptstyle{g_n}$} (A1_0);
    \path (A0_1) edge [->]node [auto] {$\scriptstyle{g_r}$} (A1_1);
    \path (A1_0) edge [->]node [auto] {$\scriptstyle{\overline{\psi}}$} (A1_1);
  \end{tikzpicture}
  \]
where $\overline{\psi}:C^n\rightarrow C^r$ is a dominant morphism with connected fibers. By Lemma \ref{fibg2} $\overline{\psi}$ must factor through a product of $r$ of the projections onto the factors, and therefore $\psi$ must factor through a forgetful morphism.\\
Finally, to get the result for dominant morphisms $C[n]\rightarrow C[r_1]\times ...\times C[r_k]$ with connected fibers it is enough to compose $\psi$ with projections onto the factor as in Corollary \ref{corfactprod}.
\end{proof}

\section{On the automorphisms of $X[n]$ and $\overline{M}_{0,n}(\mathbb{P}^N,d)$}\label{autsect}
In this section, taking advantage of the results on the fibrations in Sections \ref{LinPenc} and \ref{sectionFib}, we study the automorphisms of the Fulton-MacPherson compactification $X[n]$, and of the Kontsevich moduli space $\overline{M}_{0,n}(\mathbb{P}^N,d)$ in some significant cases.

\subsection{Groups naturally acting on $\overline{M}_{0,n}(X,\beta)$}\label{actions}
Let $X$ be a homogeneous variety. The symmetric group $S_{n}$, and the connected component of the identity $\Aut^{o}(X)$ of $\Aut(X)$ act naturally on $\overline{M}_{0,n}(X,\beta)$ by
\stepcounter{Theorem}
\begin{equation}\label{actsym}
\begin{array}{ccc}
S_{n}\times\overline{M}_{0,n}(X,\beta)& \longrightarrow & \overline{M}_{0,n}(X,\beta)\\
 (\sigma,\left[C,(x_1,...,x_n),\alpha\right]) & \longmapsto & [C,(x_{\sigma(1)},...,x_{\sigma(n)}),\alpha]
\end{array}
\end{equation}
and 
\stepcounter{Theorem}
\begin{equation}\label{actcci}
\begin{array}{ccc}
\Aut^{o}(X)\times\overline{M}_{0,n}(X,\beta)& \longrightarrow & \overline{M}_{0,n}(X,\beta)\\
 (\mu,\left[C,(x_1,...,x_n),\alpha\right]) & \longmapsto & [C,(x_{1},...,x_{1}),\mu\circ\alpha]
\end{array}
\end{equation}
These groups induce automorphisms of $\overline{M}_{0,n}(X,\beta)$, and their actions commute, that is $S_n\times \Aut^{o}(X)\subseteq \Aut(\overline{M}_{0,n}(X,\beta))$. For instance, we have the following simple result.

\begin{Proposition}
The automorphisms of $\overline{M}_{0,0}(\mathbb{P}^{2},2)$ are exactly the ones induced by automorphisms of $\mathbb{P}^{2}$, that is
$$\Aut(\overline{M}_{0,0}(\mathbb{P}^{2},2))\cong PGL(3)$$
\end{Proposition}
\begin{proof}
It is well known that the space $\overline{M}_{0,0}(\mathbb{P}^{2},2)$ is isomorphic to the space of complete conics, that is the blow-up of $\mathbb{P}^{5}$ along the Veronese surface $V\subset\mathbb{P}^{5}$ parametrizing double lines \cite[Section 0.4]{FP}. Then by \cite[Corollary 7.15]{Har} the automorphism group of $\overline{M}_{0,0}(\mathbb{P}^{2},2)$ is isomorphic to the subgroup $\Aut_V(\mathbb{P}^5)\subset PGL(6)$ of automorphisms of $\mathbb{P}^5$ stabilizing $V\cong\mathbb{P}^{2}$. To conclude it is enough to observe that $\Aut_V(\mathbb{P}^5)\cong PGL(3)$.
\end{proof}

Now, our aim is to study the connected component of the identity of $\Aut(X[n])$ and $\Aut(\overline{M}_{0,n}(X,\beta))$. The central ingredient of our argument will be the following scheme-theoretic version of Blanchard's theorem \cite[Section I.1]{Bl} due to M. Brion.

\begin{Theorem}\label{brion}\cite[Proposition 2.1]{Br}
Let $G$ be a connected group scheme, $X$ a scheme with an action of $G$, and $f:X\rightarrow Y$ a proper morphism such that $f_{*}\mathcal{O}_X \cong\mathcal{O}_Y$. Then there is a unique action of $G$ on $Y$ such that $f$ is equivariant.
\end{Theorem}

We recall the following well-known fact.

\begin{Remark}\label{flaut}
If $\phi:X\rightarrow X$ is an automorphism of any scheme then the fixed locus of $\phi$ forms a closed subscheme. So once the fixed locus includes a dense set $\mathcal{U}\subset X$, the fixed locus is the entire space, and $\phi$ is the identity. 
\end{Remark}

In order to use inductive arguments we will need the following result.

\begin{Lemma}\label{induction}
Let $X$ be a homogeneous variety. If $\Aut^{o}(\overline{M}_{0,n}(X,\beta))\cong\Aut^{o}(X)$ for an $n\geq 5$ then $\Aut^{o}(\overline{M}_{0,k}(X,\beta))\cong\Aut^{o}(X)$ for any $k\geq n$.
\end{Lemma}
\begin{proof}
Let $\phi\in \Aut^{o}(\overline{M}_{0,k+1}(X,\beta))$, with $k\geq n$, be an automorphism. We proceed by induction on $k$. By Theorem \ref{brion} with $f = \pi_i$ for $i = 1,...,k+1$, we get the following $k+1$ commutative diagrams
 \[
  \begin{tikzpicture}[xscale=3.0,yscale=-1.2]
    \node (A0_0) at (0, 0) {$\overline{M}_{0,k+1}(X,\beta)$};
    \node (A0_1) at (1, 0) {$\overline{M}_{0,k+1}(X,\beta)$};
    \node (A1_0) at (0, 1) {$\overline{M}_{0,k}(X,\beta)$};
    \node (A1_1) at (1, 1) {$\overline{M}_{0,k}(X,\beta)$};
    \path (A0_0) edge [->]node [auto] {$\scriptstyle{\phi}$} (A0_1);
    \path (A1_0) edge [->]node [auto] {$\scriptstyle{\overline{\phi}_{1}}$} (A1_1);
    \path (A0_1) edge [->]node [auto] {$\scriptstyle{\pi_{1}}$} (A1_1);
    \path (A0_0) edge [->]node [auto,swap] {$\scriptstyle{\pi_{1}}$} (A1_0);
  \end{tikzpicture}
\cdots  
    \begin{tikzpicture}[xscale=3.0,yscale=-1.2]
    \node (A0_0) at (0, 0) {$\overline{M}_{0,k+1}(X,\beta)$};
    \node (A0_1) at (1, 0) {$\overline{M}_{0,k+1}(X,\beta)$};
    \node (A1_0) at (0, 1) {$\overline{M}_{0,k}(X,\beta)$};
    \node (A1_1) at (1, 1) {$\overline{M}_{0,k}(X,\beta)$};
    \path (A0_0) edge [->]node [auto] {$\scriptstyle{\phi}$} (A0_1);
    \path (A1_0) edge [->]node [auto] {$\scriptstyle{\overline{\phi}_{k+1}}$} (A1_1);
    \path (A0_1) edge [->]node [auto] {$\scriptstyle{\pi_{k+1}}$} (A1_1);
    \path (A0_0) edge [->]node [auto,swap] {$\scriptstyle{\pi_{k+1}}$} (A1_0);
  \end{tikzpicture}
  \]
where $\overline{\phi}_i\in\Aut^{o}(\overline{M}_{0,k}(X,\beta))$ for any $i = 1,...,k+1$. Therefore, by induction hypothesis $\overline{\phi}_i$ is given by    
$$
\begin{array}{cccc}
\overline{\phi}_i: &\overline{M}_{0,k}(X,\beta)& \longrightarrow & \overline{M}_{0,k}(X,\beta)\\
      & \left[C,(x_1,...,\hat{x}_i,...,x_{k+1}),\alpha\right] & \longmapsto & \left[C,(x_1,...,\hat{x}_i,...,x_{k+1}),\mu_i\circ\alpha\right]
\end{array}
$$
with $\mu_i\in \Aut^{o}(X)$, for any $i = 1,...,k+1$. Now, let $[C,(x_1,...,x_{k+1}),\alpha]\in \overline{M}_{0,k+1}(X,\beta)$ be a general point, and let $[\Gamma,(y_1,...,y_{k+1}),\gamma] = \phi([C,(x_1,...,x_{k+1}),\alpha])$ be its image.\\
Since $\pi_i\circ \phi = \overline{\phi}_i\circ\pi_i$ we get $[C,(x_1,...,\hat{x}_i,...,x_{k+1}),\mu_i\circ\alpha] = [\Gamma,(y_1,...,\hat{y}_i,...,y_{k+1}),\gamma]$. Therefore, for any $i = 1,...,k+1$ we have an isomorphism $\tau_i:C\rightarrow \Gamma$ such that $\tau_i(x_j) = y_j$ for any $j\neq i$ and $\gamma\circ\tau_i = \mu_i\circ\alpha$.\\ 
Now, $C$ and $\Gamma$ are two smooth rational curves and since $k\geq n\geq 5$ the isomorphisms $\tau_i,\tau_j:C\rightarrow \Gamma$ coincide on at least three marked points $x_h$ with $h\neq i,j$. Therefore, $\tau_1 = \tau_2 = ... = \tau_{k+1}$ and $\mu_1\circ\alpha = \mu_2\circ\alpha = ... = \mu_{k+1}\circ\alpha$. Since $X$ is homogeneous this yields $\mu_1 = \mu_2 = ... = \mu_{k+1}$. Let us denote this automorphism of $X$ by $\mu$, and consider the morphism of groups:
$$
\begin{array}{cccc}
\chi: &\Aut^{o}(\overline{M}_{0,k+1}(X,\beta))& \longrightarrow & \Aut^{o}(X)\\
      & \phi & \longmapsto & \mu
\end{array}
$$
By Section \ref{actions} $\chi$ is surjective. Now, assume that $\mu = \chi(\phi) = Id_{X}$. Then $\overline{\phi}_i = Id_{\overline{M}_{0,k}(X,\beta)}$ for any $i = 1,...,k+1$. Since $\overline{\phi}_1 = Id_{\overline{M}_{0,k}(X,\beta)}$ the automorphism $\phi$ restricts to an automorphism of the fiber $F_1:=\pi_{1}^{-1}([C,(x_2,...,x_{k+1}),\alpha])\cong C$. Note that when $x_1$ collides with $x_2,...,x_{k+1}$ we get $k$ special points $\overline{x}_i\in F_1$ for $i = 2,...,{k+1}$, where $\overline{x}_i\in F_1$ corresponds to a stable map $[C\cup\mathbb{P}^1,(x_1,...,x_{k+1}),\alpha]$ with reducible domain, where $x_1,x_i\in\mathbb{P}^1$, and $\alpha$ contracts $\mathbb{P}^1$.\\ 
Now, since $\pi_i\circ\phi = \pi_i$ for $i = 2,...,k+1$ the automorphism $\pi_{|F_1}:F_1\rightarrow F_1$ must fix $\overline{x}_i\in F_1$ for any $i = 2,...,k+1$. Since $k \geq 5$ this yields that $\phi_{|F_1} = Id_{F_1}$. Therefore, $\phi$ restricts to the identity on the general fiber of $\pi_1$. Finally, to conclude that $\phi = Id_{\overline{M}_{0,k+1}(X,\beta)}$ it is enough to recall Remark \ref{flaut}.    
\end{proof}

\subsubsection{Automorphisms of Cartesian products}
We will need the following simple results on the automorphisms of Cartesian products.

\begin{Lemma}\label{autprod}
Let $X_1,...,X_n$ be complete varieties. Then $\Aut^{o}(X_1\times ...\times X_n)\cong \Aut^{o}(X_1)\times ...\times\Aut^{o}(X_n)$.
\end{Lemma}
\begin{proof}
The statement is trivial for $n=1$. Let $Y = X_{2}\times ...\times X_n$, then by \cite[Corollary 2.3]{Br} we have
$$\Aut^{o}(X_1\times ...\times X_n) =\Aut^{o}(X_1\times Y)\cong \Aut^{o}(X_1)\times \Aut^{o}(Y)$$
To conclude it is enough to argue by induction on $n$. 
\end{proof}

Furthermore, by Lemma \ref{fibg2} it is straightforward to compute the automorphism group of a Cartesian product of curves of genus different from one.

\begin{Lemma}\label{autcp}
Let $C_1,...,C_r$ be smooth projective curves of genus $g(C_i)\neq 1$, and let us denote by $[C_{r_1}],...,[C_{r_k}]$ the isomorphism classes of curves in $\{C_1,...,C_r\}$, where $r_i$ is the number of curves of class $[C_{r_i}]$. Then
$$\Aut(C_1\times ...\times C_r) \cong (S_{r_1}\ltimes\Aut(C_{r_1})^{r_1})\times ...\times (S_{r_k}\ltimes\Aut(C_{r_k})^{r_k}).$$
In particular, $\Aut(C^r) \cong S_{r}\ltimes \Aut(C)^r$.
\end{Lemma}
\begin{proof}
Let us fix an automorphism $\phi\in\Aut(C_1\times ...\times C_r)$. We may assume that the curves of class $[C_{r_1}]$ are $C_1,...,C_{r_1}$. Let $\pi_i:C_1\times ...\times C_r\rightarrow C_1\times ....\widehat{C}_i\times ...\times C_r$ be the projection forgetting the point on $C_i$. \\  
By Lemma \ref{fibg2} for any $i\in\{1,...,r\}$ the morphism $\pi_i\circ \phi^{-1}:C_1\times ...\times C_r\rightarrow C_1\times ....\widehat{C}_i\times ...\times C_r$ factors through a projection $\pi_{j_i}$, and hence we have the commutative diagram 
  \[
  \begin{tikzpicture}[xscale=4.5,yscale=-1.2]
    \node (A0_0) at (0, 0) {$C_1\times ....\times C_r$};
    \node (A0_1) at (1, 0) {$C_1\times ...\times C_r$};
    \node (A1_0) at (0, 1) {$C_1\times ....\widehat{C}_{j_i}\times ...\times C_r$};
    \node (A1_1) at (1, 1) {$C_1\times ....\widehat{C}_{i}\times ...\times C_r$};
    \path (A0_0) edge [->]node [auto] {$\scriptstyle{\phi^{-1}}$} (A0_1);
    \path (A1_0) edge [->]node [auto] {$\scriptstyle{\overline{\phi}}$} (A1_1);
    \path (A0_1) edge [->]node [auto] {$\scriptstyle{\pi_{i}}$} (A1_1);
    \path (A0_0) edge [->,swap]node [auto] {$\scriptstyle{\pi_{j_i}}$} (A1_0);
  \end{tikzpicture}
  \]
Note that $\phi^{-1}$ induces an isomorphism between the fiber of $\pi_i$ which is $C_i$, and the fiber of $\pi_{j_i}$ which is $C_{j_i}$. Therefore, $i\in\{1,...,r_1\}$ forces $j_i\in\{1,...,r_1\}$ as well, and we get a surjective morphism of groups
$$
\begin{array}{cccc}
\chi: &\Aut(C_1\times ...\times C_r)& \longrightarrow & S_{r_1}\\
      & \phi & \longmapsto & \sigma_{\phi}
\end{array}
$$
where $\sigma_{\phi}(i) = j_i$. Now, if $\phi$ induces the trivial permutation via $\chi$ then $\phi\in \Aut(C_{r_1})^{r_1}\times \Aut(C_{r_1+1}\times ...\times C_r)$, where the product is direct since the actions of the two groups commute. Proceeding by induction on $r$ we have $\Aut(C_{r_{1}+1}\times ...\times C_r)\cong (S_{r_2}\ltimes\Aut(C_{r_2})^{r_2})\times ...\times (S_{r_k}\ltimes\Aut(C_{r_k})^{r_k})$, and hence 
$$\Aut(C_1\times ...\times C_r)\cong S_{r_1}\ltimes (\Aut(C_{r_1})^{r_1}\times (S_{r_2}\ltimes\Aut(C_{r_2})^{r_2})\times ...\times (S_{r_k}\ltimes\Aut(C_{r_k})^{r_k}))$$
To conclude it is enough to observe that the action of $S_{r_1}$ commutes with the action of $S_{r_i}\ltimes\Aut(C_{r_i})^{r_i}$ for any $i=2,...,k$, but does not commute with the action of $\Aut(C_{r_1})^{r_1}$.
\end{proof}

\begin{Remark}
In the proof of Lemma \ref{autcp} we considered $\pi_i\circ \phi^{-1}$ instead of $\pi_i\circ \phi$ in order to make the map $\chi$ a morphism of groups. For the same reason, in similar settings, we will consider $\phi^{-1}$ instead of $\phi$ several times in the rest of the paper.
\end{Remark}

As an application of Theorem \ref{brion} we get the following result. 

\begin{Proposition}\label{connFM}
If either $n\neq 2$ or $\dim(X)\geq 2$, then the connected component of the identity of $\Aut(X[n])$ is isomorphic to the connected component of the identity of $\Aut(X)$, that is
$$\Aut^{o}(X[n])\cong \Aut^{o}(X)$$
for any $n$.
\end{Proposition}
\begin{proof}
Let $g_n:X[n]\rightarrow X^n$ be the blow-up morphism in Proposition \ref{sym}. Since $g_n$ is birational we have $g_{n*}\mathcal{O}_{X[n]} \cong\mathcal{O}_{X^n}$, and we may apply Theorem \ref{brion} with $f = g_n$ and $G = \Aut^{o}(X[n])$. Indeed, by Theorem \ref{brion} for any automorphism $\phi\in \Aut^{o}(X[n])$ there exists an automorphism $\overline{\phi}\in \Aut^{o}(X^n)$ such that the following diagram
 \[
  \begin{tikzpicture}[xscale=2.9,yscale=-1.2]
    \node (A0_0) at (0, 0) {$X[n]$};
    \node (A0_1) at (1, 0) {$X[n]$};
    \node (A1_0) at (0, 1) {$X^n$};
    \node (A1_1) at (1, 1) {$X^n$};
    \path (A0_0) edge [->]node [auto] {$\scriptstyle{\phi}$} (A0_1);
    \path (A1_0) edge [->]node [auto] {$\scriptstyle{\overline{\phi}}$} (A1_1);
    \path (A0_1) edge [->]node [auto] {$\scriptstyle{g_{n}}$} (A1_1);
    \path (A0_0) edge [->]node [auto,swap] {$\scriptstyle{g_{n}}$} (A1_0);
  \end{tikzpicture}
  \]
is commutative. Now, let $x = (x,...,x)\in X^n$ be a point in the small diagonal $\Delta_{1,...,n}$, and assume that $\overline{\phi}(x) = y\notin\Delta_{1,...,n}$. Let $F_x$ and $F_y$ be the fibers of $g_{n}$ over $x$ and $y$ respectively. Then $\phi$ restricts to an isomorphism $\phi_{|F_x}:F_x\rightarrow F_y$. On the other hand by Proposition \ref{sym} we know that $\dim(F_x) = (n-1)\dim(X)-1$, while $y\notin\Delta_{1,...,n}$ yields $\dim(F_y)< (n-1)\dim(X)-1$. A contradiction. Therefore, $\overline{\phi}$ restricts to an automorphism of $\Delta_{1,...,n}\cong X$, and we get a morphism of groups:
$$
\begin{array}{cccc}
\chi: &\Aut^{o}(X[n]) & \longrightarrow & \Aut^o(X)\\
      & \phi & \longmapsto & \overline{\phi}_{|\Delta_{1,...,n}}
\end{array}
$$
Furthermore, by Lemma \ref{autprod} $\overline{\phi}$ comes from the diagonal action of $\Aut^{o}(X)$. Now, all the subvarieties of $X^n$ blown-up in the construction of Proposition \ref{sym} are stabilized by the diagonal action of $\Aut^{o}(X)$ on $X^n$. Therefore, by \cite[Corollary 7.15]{Har} this action lifts to an action of $\Aut^{o}(X)$ on $X[n]$, and the morphism $\chi$ is surjective.\\ 
Finally, let $\phi\in\Aut^{o}(X[n])$ such that $\chi(\phi) = \overline{\phi}_{|\Delta_{1,...,n}} = Id_{\Delta_{1,...,n}}$. Now $\overline{\phi}\in \Aut^{o}(X^n)$ and by Lemma \ref{autprod} we have $\Aut^{o}(X^n)\cong \Aut^{o}(X)^n$. Therefore, we may write $\overline{\phi} = (\overline{\phi}_1,...,\overline{\phi}_n)$ where $\overline{\phi}_i:X\rightarrow X$ is an automorphism of $X$ for any $i = 1,...,n$. Furthermore for any $x\in X$ we have $\overline{\phi}(x,...,x) = (\overline{\phi}_1(x),...,\overline{\phi}_n(x)) =  Id_{\Delta_{1,...,n}}(x,...,x) = (x,...,x)$ that is $\overline{\phi}_i(x) = x$. Then $\overline{\phi}_i = Id_{X}$ for any $i = 1,...,n$ and $\overline{\phi} = Id_{X^n}$.

Therefore, the automorphism $\phi$ restricts to an automorphism of a general fiber of $g_n$. On the other hand, since $g_n$ is birational such general fiber is a point. That is $\phi$ restricts to the identity on $X[n]\setminus\bigcup_{2\leq |S|\leq n}g_{n}^{-1}(\Delta_S)$. By Remark \ref{flaut} we conclude that $\phi$ is the identity, and $\chi$ is injective.
\end{proof}

By Proposition \ref{connFM} we have that if $n\neq 2$ then $\Aut^{o}(\mathbb{P}^1[n])\cong PGL(2)$. Thanks to the main result on dominant morphisms from $\mathbb{P}^1[n]$ to $\mathbb{P}^1$ in Section \ref{LinPenc} we have the following stronger result.

\begin{Theorem}\label{autFMP1}
The automorphism group of the Fulton-MacPherson compactification $\mathbb{P}^1[n]$ is given by 
$$\Aut(\mathbb{P}^1[n])\cong S_n\times PGL(2)$$
if $n\neq 2$. Furthermore, $\Aut(\mathbb{P}^1[2])\cong S_2\ltimes (PGL(2)\times PGL(2))$.
\end{Theorem}
\begin{proof}
Since $\mathbb{P}^1[1]\cong\mathbb{P}^1$ the statement is trivial for $n = 1$. If $n = 2$ it follows from Lemma \ref{autcp}. Now, let us consider the case $n\geq 3$. As usual we take an automorphism $\phi\in \Aut(\mathbb{P}^1[n])$, and for any $i\in\{1,...,n\}$ we consider the composition $ev_i\circ\phi^{-1}$. By Proposition \ref{propfact} the morphism $ev_i\circ\phi^{-1}$ factors either through an evaluation map $ev_{j_i}$ or through a forgetful morphism $\pi_{i_1,...,i_{n-4}}\circ \rho$ of type (\ref{forg}).\\
Let us assume that $ev_i\circ\phi^{-1}$ factors through $\pi_{i_1,...,i_{n-4}}\circ \rho$. Then we have the following commutative diagram
  \[
  \begin{tikzpicture}[xscale=3.5,yscale=-1.2]
    \node (A0_0) at (0, 0) {$\mathbb{P}^1[n]$};
    \node (A0_1) at (1, 0) {$\mathbb{P}^1[n]$};
    \node (A1_0) at (0, 1) {$\overline{M}_{0,4}\cong\mathbb{P}^1$};
    \node (A1_1) at (1, 1) {$\mathbb{P}^1$};
    \path (A0_0) edge [->]node [auto] {$\scriptstyle{\phi^{-1}}$} (A0_1);
    \path (A1_0) edge [->]node [auto] {$\scriptstyle{\overline{\phi}}$} (A1_1);
    \path (A0_1) edge [->]node [auto] {$\scriptstyle{ev_i}$} (A1_1);
    \path (A0_0) edge [->,swap]node [auto] {$\scriptstyle{\pi_{i_1,...,i_{n-4}}\circ \rho}$} (A1_0);
  \end{tikzpicture}
  \]
Now, let $p\in\overline{M}_{0,4}$ be the point corresponding to the isomorphism class of a curve $[C = C_1\cup C_2,(x_1,...,x_4)]$, where $C_1,C_2$ are smooth rational curves intersecting in one node, $x_1,x_2\in C_1$ and $x_3,x_4\in C_2$. Then $\pi_{i_1,...,i_{n-4}}^{-1}(p)$ is the union of $\sum_{k=0}^{n-4}\binom{n-4}{k}$ irreducible boundary divisors each one determined by a subset of $\{i_1,...,i_{n-4}\}$ labeling the marked points on $C_1$. On the other hand, the general fiber of $\pi_{i_1,...,i_{n-4}}\circ \rho$ is irreducible. Now, to get a contradiction it is enough to recall that by Lemma \ref{isofib} all the fibers of $ev_i$ are isomorphic.\\
Therefore $ev_i\circ\phi^{-1}$ must factor through another evaluation map $ev_{j_i}$, and this yields a surjective morphism of groups
$$
\begin{array}{cccc}
\chi_n: &\Aut(\mathbb{P}^1[n])& \longrightarrow & S_n\\
      & \phi & \longmapsto & \sigma_{\phi}
\end{array}
$$
where $\sigma_{\phi}(i) = j_i$. Since $ev_i\circ\phi^{-1} = \overline{\phi}_i\circ ev_{j_i}$ with $\overline{\phi}_i\in PGL(2)$, we have that 
$$(ev_1\times ...\times ev_n)\circ\phi^{-1} = (\overline{\phi}_1\times...\times \overline{\phi}_n)\circ (ev_{\sigma(1)}\times...\times ev_{\sigma(n)})$$
and the commutative diagram
  \[
  \begin{tikzpicture}[xscale=3.5,yscale=-1.2]
    \node (A0_0) at (0, 0) {$\mathbb{P}^1[n]$};
    \node (A0_1) at (1, 0) {$\mathbb{P}^1[n]$};
    \node (A1_0) at (0, 1) {$(\mathbb{P}^1)^n$};
    \node (A1_1) at (1, 1) {$(\mathbb{P}^1)^n$};
    \path (A0_0) edge [->]node [auto] {$\scriptstyle{\phi^{-1}}$} (A0_1);
    \path (A0_0) edge [->,swap]node [auto] {$\scriptstyle{ev_{\sigma(1)}\times...\times ev_{\sigma(n)}}$} (A1_0);
    \path (A0_1) edge [->]node [auto] {$\scriptstyle{ev_1\times ...\times ev_n}$} (A1_1);
    \path (A1_0) edge [->]node [auto] {$\scriptstyle{\overline{\phi}_1\times ...\times \overline{\phi}_n}$} (A1_1);
  \end{tikzpicture}
  \]
Note that by Proposition \ref{sym} both $ev_1\times ...\times ev_n$ and $ev_{\sigma(1)}\times ...\times ev_{\sigma(n)}$ are blow-ups of the diagonals on $(\mathbb{P}^1)^n$ in order of increasing dimension. Arguing exactly as in the proof of Proposition \ref{connFM} we see that $\overline{\phi}_1\times ...\times \overline{\phi}_n\in \Aut((\mathbb{P}^1)^n)$ must preserve the diagonal $\Delta_{1,...,n}\cong\mathbb{P}^1$, and therefore it yields an automorphism $\mu_{\phi}\in PGL(2)$. This induces a morphism of groups 
$$
\begin{array}{cccc}
\overline{\chi}_n: &\Aut(\mathbb{P}^1[n])& \longrightarrow & S_n\times PGL(2)\\
      & \phi & \longmapsto & (\sigma_{\phi},\mu_{\phi})
\end{array}
$$
which, by Section \ref{actions} is surjective. Now assume that $\overline{\chi}_n(\phi)$ is the identity. Then arguing as in the proof of Proposition \ref{connFM} we have that $\phi^{-1}$ stabilizes the general fiber of $ev_1\times...\times ev_n$. On the other hand  $ev_1\times...\times ev_n$ is birational and hence $\phi^{-1}$ restricts to the identity on an open subset of $\mathbb{P}^1[n]$. To conclude it is enough to observe that Remark \ref{flaut} forces $\phi^{-1} = \phi = Id_{\mathbb{P}^1[n]}$.
\end{proof}

With similar arguments we can attack the automorphism group of $\overline{M}_{0,n}(\mathbb{P}^N,1)$. Note that since the degree of the map is one there is a natural $PGL(2)$ action on $\overline{M}_{0,n}(\mathbb{P}^N,1)$. Let $[C,(x_1,...,x_n),\alpha]\in \overline{M}_{0,n}(\mathbb{P}^N,1)$ be a point, and let $\nu\in PGL(2)$.\\ 
There exists a unique component of $C$, say $C_1$, on which $\alpha$ has degree one. We consider the stable map $[\Gamma,(y_1,...,y_n),\alpha]\in \overline{M}_{0,n}(\mathbb{P}^N,1)$ obtained by acting with $\nu$ on $C_1$, and with the identity on the remaining components of $C$.\\
More precisely, let $p = C\cap \overline{C\setminus C_1}$, and let $x_{i_1},...,x_{i_k}$ be the marked points lying on $C_1$. We consider the pointed curve $(C_1,(\nu(x_{i_1}),...,\nu(x_{i_k}))$ and we attach to it a copy of $\overline{C\setminus C_1}$ at $\nu(p)$. Letting the map $\alpha$ unvaried this gives us the stable map $[\Gamma,(y_1,...,y_n),\alpha]$, and therefore an action
\stepcounter{Theorem}
\begin{equation}\label{actiondeg1}
\begin{array}{ccc}
PGL(2)\times\overline{M}_{0,n}(\mathbb{P}^N,1)& \longrightarrow & \overline{M}_{0,n}(\mathbb{P}^N,1)\\
\end{array}
\end{equation}
which is trivial when $n=0$, and coincides with (\ref{actcci}) when $N=1$. Indeed, when $n = 0$ we have $\overline{M}_{0,0}(\mathbb{P}^N,1)\cong \mathbb{G}(1,N)$, the Grassmannian of lines in $\mathbb{P}^N$, and \cite[Theorem 1.1]{Co} yields $\Aut^{o}(\mathbb{G}(1,N))\cong PGL(N+1)$.\\
Furthermore, if $n = 2$ we may define an action of $PGL(2)\times PGL(2)$ on $\overline{M}_{0,2}(\mathbb{P}^N,1)$. Indeed, given $(\nu_1,\nu_2)\in PGL(2)\times PGL(2)$ we can map a general point $[\mathbb{P}^1,(x_1,x_2), \alpha]\in \overline{M}_{0,2}(\mathbb{P}^N,1)$ to $[\mathbb{P}^1,(\nu_1(x_1),\nu_2(x_2)),\alpha]$. Note that a boundary point in $\overline{M}_{0,2}(\mathbb{P}^N,1)$ necessarily represents a stable map of the form $[C_1\cup C_2,(x_1,x_2),\alpha]$, where $x_1,x_2\in C_2$ and $\alpha$ has degree one on $C_1$. Now, we consider the curve $C_1$ with $x_1$ and $x_2$ collapsed in a point $x=x_1=x_2$. If $\nu_1(x)\neq \nu_2(x)$ then the image of $[C_1\cup C_2,(x_1,x_2),\alpha]$ will be $[C_1,(\nu_1(x),\nu_2(x)), \alpha]$. If $\nu_1(x)=\nu_2(x)$ then the image will be $[C_1\cup C_2,(y_1,y_2), \alpha]$ where $C_2$ is a smooth rational curve with two marked points attached to $C_1$ at the point $\nu_1(x)=\nu_2(x)$. Finally, if we have a stable map of the type $[\mathbb{P}^1,(x_1,x_2), \alpha]$, and $\nu_1(x_1) = \nu_2(x_2)$ then we map such a stable map to $[C_1\cup C_2,(y_1,y_2),\alpha]$, where $y_1,y_2\in C_2$ and $C_2$ is attached to $C_1$ at $\nu_1(x_1) = \nu_2(x_2)$. In this way we get a well-defined regular action
\stepcounter{Theorem}
\begin{equation}\label{actiondeg1n2}
\begin{array}{ccc}
(PGL(2)\times PGL(2))\times\overline{M}_{0,2}(\mathbb{P}^N,1)& \longrightarrow & \overline{M}_{0,2}(\mathbb{P}^N,1)\\
\end{array}
\end{equation}
which coincides with the action in Theorem \ref{autFMP1} when $N=1$.

\begin{Proposition}\label{propGrass}
If $N\geq 2$ and $n\geq 1$ then the connected component of the identity of the automorphism group of $\overline{M}_{0,n}(\mathbb{P}^N,1)$ is given by
$$\Aut^{o}(\overline{M}_{0,n}(\mathbb{P}^N,1))\cong PGL(2)\times PGL(N+1)$$
for any $n\neq 2$. Furthermore, $\Aut^{o}(\overline{M}_{0,2}(\mathbb{P}^N,1))\cong PGL(2)\times PGL(2)\times PGL(N+1)$.
\end{Proposition}
\begin{proof}
Let us consider the forgetful morphism $\pi:\overline{M}_{0,n}(\mathbb{P}^N,1)\rightarrow \overline{M}_{0,0}(\mathbb{P}^N,1)\cong\mathbb{G}(1,N)$, where $\mathbb{G}(1,N)$ is the Grassmannian of lines in $\mathbb{P}^N$. By \cite[Theorem 1.1]{Co} we have $\Aut^{o}(\mathbb{G}(1,N))\cong PGL(N+1)$, and Theorem \ref{brion} yields a surjective morphism of groups
$$
\begin{array}{cccc}
\chi: & \Aut^{o}(\overline{M}_{0,n}(\mathbb{P}^N,1)) & \longrightarrow & PGL(N+1)\\
      & \phi & \longmapsto & \overline{\phi}
\end{array}
$$
Note that $\pi:\overline{M}_{0,n}(\mathbb{P}^N,1)\rightarrow \overline{M}_{0,0}(\mathbb{P}^N,1)\cong\mathbb{G}(1,N)$ is a fibration with the Fulton-MacPherson compactification $\mathbb{P}^1[n]$ as the fiber.\\
Now, if $\chi(\phi)$ is the identity then $\phi$ induces an automorphism of the fiber $\pi^{-1}([C,\alpha])\cong\mathbb{P}^{1}[n]$ lying in $\Aut^{o}(\mathbb{P}^1[n])$. By Theorem \ref{autFMP1} we have that $\Aut^{o}(\mathbb{P}^1[n])\cong PGL(2)$ if $n\neq 2$, and $\Aut^{o}(\mathbb{P}^1[2])\cong PGL(2)\times PGL(2)$. Note that in any case $\Aut^{o}(\mathbb{P}^1[n])$ may be embedded in $\Aut^{o}(\overline{M}_{0,n}(\mathbb{P}^N,1))$ via the actions (\ref{actiondeg1}) and (\ref{actiondeg1n2}) in the cases $n\neq 2$ and $n=2$ respectively. Therefore, in any case we get an exact sequence 
$$0\mapsto \Aut^{o}(\mathbb{P}^1[n])\rightarrow  \Aut^{o}(\overline{M}_{0,n}(\mathbb{P}^N,1))\rightarrow PGL(N+1)\mapsto 0$$
To conclude, it is enough to observe that by (\ref{actcci}) the sequence above splits, and that the action (\ref{actcci}) commutes with both the actions (\ref{actiondeg1}) and (\ref{actiondeg1n2}).
\end{proof}

\subsubsection{Automorphisms of $C[n]$ and $\overline{M}_{g,n}$}
We can extend the main ideas in the computation of $\Aut(\mathbb{P}^1[n])$ to $C[n]$ where $C$ is a smooth curve of genus $g(C)\geq 2$.

\begin{Proposition}\label{druel2}
Let $X$ be a smooth projective variety with nef canonical divisor, and let $\Aut_{\Delta}(X^n)\subseteq \Aut(X^n)$ be the subgroup of automorphisms stabilizing the union of all the diagonals of codimension greater than one in $X^n$. Then we have an isomorphism $\Aut(X[n])\cong \Aut_{\Delta}(X^n)$.\\
In particular, if $X\cong C$ is a curve of genus $g(C)\geq 2$ then
$$\Aut(C[n])\cong S_n\times \Aut(C)$$
if $n\neq 2$, and $\Aut(C[2])\cong S_2\ltimes (\Aut(C)\times\Aut(C))$.
\end{Proposition}
\begin{proof}
The case $\dim(X) = 1$ and $n=2$ is in Lemma \ref{autcp}. Let $g_n:X[n]\rightarrow X^n$ be the blow-up morphism in Proposition \ref{sym}. For any $S = \{i_1,...,i_s\}\subset\{1,...,n\}$ we denote by $\Delta_S = \{(x_1,...,x_{n})\in X^{n}\:|\: x_{i_1} = ... = x_{i_s}\}$ the corresponding diagonal, and by $E_S\subset X[n]$ the exceptional divisor over $\Delta_S$. If $\dim(X)\geq 2$ then the canonical divisor of $X[n]$ is given by 
$$K_{X[n]} = g_n^{*}K_{X^n}+\sum_{2\leq |S|\leq n}((s-1)\dim(X)-1)E_S.$$
When $\dim(X) = 1$ and $n\geq 3$ the sum must be taken on the subsets $S$ such that $3\leq |S|\leq n$. However, this will not be relevant in our argument.\\ 
Let $\phi\in\Aut(X[n])$ be an automorphism, and let us assume that the exceptional divisor $E_S$ is not mapped to an exceptional divisor of $g_n$ via $\phi$. Let $C$ be a general rational curve in $E_S$ contracted by $g_n$, then $\phi(C)$ can not be contained in any exceptional divisor of $g_n$ otherwise $\phi(E_S)$ would be contained in such an exceptional divisor as well. This yields 
$$(\sum_{2\leq |S|\leq n}((s-1)\dim(X)-1)E_S)\cdot \phi(C)\geq 0$$
Furthermore, since $K_{X^n}$ is nef we have $g_n^{*}K_{X^n}\cdot\phi(C) = K_{X^n}\cdot g_{n*}\phi(C)\geq 0$, and hence $K_{X[n]}\cdot \phi(C)\geq 0$.\\
On the other hand, since $\phi$ is an automorphism we have
$$K_{X[n]}\cdot \phi(C) = \phi^{*}K_{X[n]}\cdot C = K_{X[n]}\cdot C$$
and since $g_n^{*}K_{X^n}\cdot C = 0$ we get
$$K_{X[n]}\cdot C = \sum_{2\leq |S|\leq n}((s-1)\dim(X)-1)E_S\cdot C < 0$$
A contradiction. Therefore, $\phi_{|E_{S}}$ defines an isomorphism between $E_{S}$ and an exceptional divisor $E_{S^{'}}$. Let us consider the restrictions of the blow-up morphism $g_{n|E_S}:E_S\rightarrow \Delta_S$, and $g_{n|E_{S^{'}}}:E_{S^{'}}\rightarrow \Delta_{S'}$.\\ 
Now, let $y\in \Delta_{S^{'}}$ be a general point, and $q\in g_{n|E_{S^{'}}}^{-1}(y)$ a general point in the fiber $F_q$ of $g_{n|E_{S^{'}}}$ over $y$. Let $F_p$ be the fiber of $ g_{n|E_S}$ through $p = \phi^{-1}(q)$. Consider a rational curve $C\subseteq F_p$ passing through $p$, then $\phi(C)$ passes through $q\in F_q$. Assume that $\phi(C)\nsubseteq F_q$. Then $g_n(\phi(C))\subset \Delta_{S^{'}}$ is a rational curve through $y\in \Delta_{S^{'}}$. This means that $\Delta_{S^{'}}$ is uniruled. On the other hand $\Delta_{S^{'}}\cong X^{n-|S^{'}|+1}$, and by hypothesis $K_X$ is nef. A contradiction.\\
We conclude that $\phi_{|E_S}:E_S\rightarrow E_{S^{'}}$ maps isomorphically fibers of $g_{n|E_{S}}$ to fibers of $g_{n|E_{S^{'}}}$. In particular $|S| = |S^{'}|$. Hence $\phi$ induces an automorphism $\overline{\phi}$ fitting in the following commutative diagram
  \[
  \begin{tikzpicture}[xscale=1.9,yscale=-1.2]
    \node (A0_0) at (0, 0) {$X[n]$};
    \node (A0_1) at (1, 0) {$X[n]$};
    \node (A1_0) at (0, 1) {$X^n$};
    \node (A1_1) at (1, 1) {$X^n$};
    \path (A0_0) edge [->]node [auto] {$\scriptstyle{\phi}$} (A0_1);
    \path (A1_0) edge [->]node [auto] {$\scriptstyle{\overline{\phi}}$} (A1_1);
    \path (A0_1) edge [->]node [auto] {$\scriptstyle{g_n}$} (A1_1);
    \path (A0_0) edge [->,swap]node [auto] {$\scriptstyle{g_n}$} (A1_0);
  \end{tikzpicture}
  \]
and furthermore $\overline{\phi}$ maps isomorphically any diagonal $\Delta_S$ to a diagonal of the same dimension, that is $\overline{\phi}\in \Aut_{\Delta}(X^n)$. This yields a morphism of groups
$$
\begin{array}{ccc}
\Aut(X[n])& \longrightarrow & \Aut_{\Delta}(X^n)\\
       \phi & \longmapsto & \overline{\phi}
\end{array}
$$
which clearly is an isomorphism. Finally, if $X\cong C$ is a curve of genus $g(C)\geq 2$ and $n\neq 2$ the statement follows from Lemma \ref{autcp}.   
\end{proof}

\begin{Remark}\label{GL2}
Proposition \ref{druel2} does not hold if $C$ has genus one. For instance, the group $GL(2,\mathbb{Z})$ of matrices with integer entries and determinant plus or minus one acts on $C\times C$ via
$$
\begin{array}{ccc}
GL(2,\mathbb{Z})\times (C\times C)& \longrightarrow & C\times C\\
 \left(\left(\begin{array}{cc}
 a_{1,1} & a_{1,2} \\ 
 a_{2,1} & a_{2,2}
 \end{array}\right) ,(x_1,x_2)\right) & \longmapsto & (a_{1,1}x_1\cdot a_{1,2}x_2,a_{2,1}x_1\cdot a_{2,2}x_2)
\end{array}
$$
where $x\cdot y$, with $x,y\in C$, stands for the multiplication of the group law on $C$.
\end{Remark}

Now let $X = C_1\times ...\times C_r$ be a product of curves. As in Lemma \ref{autcp} we denote by $[C_{r_1}],...,[C_{r_k}]$ the isomorphism classes of curves in $\{C_1,...,C_r\}$, where $r_i$ is the number of curves of class $[C_{r_i}]$. Let us consider the product $X^n$, and let $(x_1^j,...,x_r^j)$ be the coordinates on the $j$-th copy of $X$, so that a point in $X^n$ is given by $(x_i^j)$ with $i = 1,...,r$, $j = 1,...,n$. We take into account the three following group actions on $X^n$:
\stepcounter{Theorem}
\begin{equation}\label{Sr}
\begin{array}{ccc}
(S_{r_1}\times ...\times S_{r_k})\times X^n& \longrightarrow & X^n\\
 (\sigma,(x_i^j)) & \longmapsto & ((x_{\sigma(1)}^1,...,x_{\sigma(r)}^1),...,(x_{\sigma(1)}^n,...,x_{\sigma(r)}^n))
\end{array}
\end{equation}
where $\sigma$ must be interpreted as a permutation $\sigma\in S_{r_1}\times ...\times S_{r_k}\subseteq S_r$.
\stepcounter{Theorem}
\begin{equation}\label{SnProd}
\begin{array}{ccc}
S_n^{r}\times X^n& \longrightarrow & X^n\\
 (\sigma_1,...,\sigma_r,(x_i^j)) & \longmapsto & ((x^{\sigma_1(1)}_1,...,x^{\sigma_r(1)}_r),...,(x^{\sigma_1(n)}_1,...,x^{\sigma_r(n)}_r))
\end{array}
\end{equation}
\stepcounter{Theorem}
\begin{equation}\label{autcur}
\begin{array}{ccc}
\bigoplus_{i=1}^n\Aut(C_i)\times X^n& \longrightarrow & X^n\\
 (\alpha_1,...,\alpha_r,(x_i^j)) & \longmapsto & ((\alpha_1(x^{1}_1),...,\alpha_r(x_{r}^{1})),...,(\alpha_1(x^{n}_1),...,\alpha_r(x_{r}^{n})))
\end{array}
\end{equation}

\begin{Proposition}\label{druel1}
Let $X = C_1\times ...\times C_r$ be a product of curves with $g(C_i)\geq 2$ for any $i = 1,...,r$, and let $[C_{r_1}],...,[C_{r_k}]$ be the isomorphism classes of curves in $\{C_1,...,C_r\}$, where $r_i$ is the number of curves of class $[C_{r_i}]$. If $n\neq 2$ then
$$\Aut(X[n])\cong S_n\times((S_{r_1}\ltimes \Aut(C_{r_1})^{r_1})\times ...\times (S_{r_k}\ltimes \Aut(C_{r_k})^{r_k}))\cong S_n\times \Aut(X)$$
while if $n = 2$ and $r\geq 2$ we have
$$\Aut(X[2])\cong S_2^r\ltimes ((S_{r_1}\ltimes \Aut(C_{r_1})^{r_1})\times ...\times (S_{r_k}\ltimes \Aut(C_{r_k})^{r_k})) \cong S_2^r\ltimes \Aut(X)$$
Finally, if $n = 2$ and $r=1$ then $X = C_1$, and $\Aut(C_1[2])\cong S_2\ltimes(\Aut(C_1)\times \Aut(C_1))$.
\end{Proposition}
\begin{proof}
The case $n = 2$, $r = 1$ is just the last part of Proposition \ref{druel2}. Now, by Proposition \ref{druel2} we may identify $\Aut(X[n])$ with the subgroup $\Aut_{\Delta}(X^n)\subseteq \Aut(X^n)$ of automorphisms stabilizing the union of all the diagonals of codimension greater than one in $X^n$.\\
Note that $X^n$ is just a product of curves, and by Lemma \ref{autcp} we know the structure of its automorphism group. The actions (\ref{Sr}) and (\ref{autcur}) yield an injective morphism $i:(S_{r_1}\ltimes \Aut(C_{r_1})^{r_1})\times ...\times (S_{r_k}\ltimes \Aut(C_{r_k})^{r_k})\hookrightarrow \Aut_{\Delta}(X^n)$. Now, by Lemma \ref{autcp} $\coker(i)$ is forced to be a subgroup of the group $S_n^r$ in (\ref{SnProd}).\\
If $n = 2$, $r\geq 2$. Then any automorphism in (\ref{SnProd}) preserves the diagonal $\Delta_{1,2}$, that is $\coker(i)\cong S_2^r$. To conclude it is enough to observe that (\ref{SnProd}) induces a section $S_2^r\rightarrow \Aut_{\Delta}(X^n)$.\\
If $n\neq 2$ then $(\sigma_1,...,\sigma_r)$ preserves the union of all the diagonals $\Delta\subset X^n$ if and only if $\sigma_1 = ... = \sigma_r$. This yields that $\coker(i)\cong S_n$ is given by the diagonal action of $S_n$ in (\ref{SnProd}). Again we have a section $S_n\rightarrow  \Aut_{\Delta}(X^n)$, and to conclude it is enough to observe the actions of $S_n$ and $(S_{r_1}\ltimes \Aut(C_{r_1})^{r_1})\times ...\times (S_{r_k}\ltimes \Aut(C_{r_k})^{r_k})$ commute.
\end{proof}

Let $\overline{M}_{g,n}$ be the Deligne-Mumford compactification of the moduli space $M_{g,n}$ parametrizing smooth genus $g$ curves with $n$ marked points. Thanks to Proposition \ref{druel2} we can provide a simple proof of the main theorem on $\Aut(\overline{M}_{g,n})$ in \cite{Ma} when $g\geq 3$.

\begin{Corollary}\label{mgn}
If $g\geq 3$ then $\Aut(\overline{M}_{g,n})\cong S_n$ for any $n\geq 1$.
\end{Corollary}
\begin{proof}
Let $\phi\in \Aut(\overline{M}_{g,n})$ be an automorphism. By \cite[Theorem 0.9]{GKM} for any forgetful morphism $\pi_i:\overline{M}_{g,n}\rightarrow\overline{M}_{g,n-1}$ the composition $\pi_i\circ\phi^{-1}:\overline{M}_{g,n}\rightarrow\overline{M}_{g,n-1}$ factors though a forgetful morphism $\pi_{\sigma(i)}:\overline{M}_{g,n}\rightarrow\overline{M}_{g,n-1}$. Therefore, we get the following commutative diagram  
 \[
  \begin{tikzpicture}[xscale=2.5,yscale=-1.2]
    \node (A0_0) at (0, 0) {$\overline{M}_{g,n}$};
    \node (A0_1) at (1, 0) {$\overline{M}_{g,n}$};
    \node (A1_0) at (0, 1) {$\overline{M}_{g}$};
    \node (A1_1) at (1, 1) {$\overline{M}_{g}$};
    \path (A0_0) edge [->]node [auto] {$\scriptstyle{\phi^{-1}}$} (A0_1);
    \path (A0_0) edge [->,swap]node [auto] {$\scriptstyle{\pi_{\sigma_{\phi}(1)}\times ...\times \pi_{\sigma_{\phi}(n)}}$} (A1_0);
    \path (A0_1) edge [->]node [auto] {$\scriptstyle{\pi_1\times ...\times \pi_n}$} (A1_1);
    \path (A1_0) edge [->]node [auto] {$\scriptstyle{\overline{\phi}}$} (A1_1);
  \end{tikzpicture}
  \]
and the surjective morphism of groups
$$
\begin{array}{cccc}
\chi: &\Aut(\overline{M}_{g,n})& \longrightarrow & S_{n}\\
      & \phi & \longmapsto & \sigma_{\phi}
\end{array}
$$
Let $[C]\in \overline{M}_{g}$ be a general point, and let $[\Gamma] = \overline{\phi}([C])$. Note that since both $C$ and $\Gamma$ have a trivial automorphism group, the fibers of $\pi_{\sigma_{\phi}(1)}\times ...\times \pi_{\sigma_{\phi}(n)}$ over $[C]$ and of $\pi_1\times ...\times \pi_n$ over $[\Gamma]$ are nothing but the Fulton-MacPherson compactifications $C[n]$ and $\Gamma[n]$ respectively. Therefore, $\phi^{-1}$ induces an isomorphism between $C[n]$ and $\Gamma[n]$, and this yields $C\cong \Gamma$, and $\overline{\phi} = Id_{\overline{M}_{g}}$. Now, if $\chi(\sigma)$ is the trivial permutation then $\phi$ restricts to an automorphism of the general fiber $C[n]$ of $\pi_1\times ...\times \pi_n$.\\ 
By Proposition \ref{druel2} $\phi_{|C[n]}$ acts as a combination of a permutation and an automorphism of $C$. On the other hand, since $\phi\in \ker(\chi)$ the permutation must be trivial, and since $C$ is a general curve of genus $g\geq 3$ then $\Aut(C)$ is trivial as well.\\
This means that $\phi$ restricts to the identity on the general fiber of $\pi_1\times ....\times \pi_n$. To conclude it is enough to recall Remark \ref{flaut}.
\end{proof}

\subsection{Kontsevich spaces parametrizing rational normal curves}
Let us consider the Kontsevich space $\overline{M}_{0,n+3}(\mathbb{P}^{n},n)$ parametrizing degree $n$ rational normal curves in $\mathbb{P}^n$. It is well known that through $n+3$ general points in $\mathbb{P}^n$ there is a unique rational normal curve of degree $n$, that is the evaluation morphism
$$ev:=ev_1\times ...\times ev_{n+3}:\overline{M}_{0,n+3}(\mathbb{P}^n,n)\rightarrow(\mathbb{P}^n)^{n+3}$$
is birational. Therefore, we may adapt the argument in the proof of Proposition \ref{connFM} to prove that the connected component of the identity of $\Aut(\overline{M}_{0,n+3}(\mathbb{P}^n,n))$ is isomorphic to $PGL(n+1)$. However, a little improvement is at hand if we take into account Kapranov's construction of the moduli space of Deligne-Mumford stable $n$-pointed rational curves $\overline{M}_{0,n}$. Thanks to \cite[Theorem 0.1]{Ka} we may consider $\overline{M}_{0,n+2}(\mathbb{P}^n,n)$ instead of $\overline{M}_{0,n+3}(\mathbb{P}^n,n)$ in order to prove the following result.

\begin{Proposition}\label{connRNC}
For any $n\geq 3$ and $k\geq n+2$ we have $\Aut^{o}(\overline{M}_{0,k}(\mathbb{P}^n,n))\cong PGL(n+1)$.
\end{Proposition}
\begin{proof}
By Theorem \cite[Theorem 0.1]{Ka} the morphism
$$\rho\times ev_{1}\times...\times ev_{n+2}:\overline{M}_{0,n+2}(\mathbb{P}^{n},n)\rightarrow\overline{M}_{0,n+2}\times(\mathbb{P}^{n})^n$$
is an isomorphism on the open subset of $(\mathbb{P}^{n})^n$ parametrizing points in linear general position, and the projection on $(\mathbb{P}^{n})^n$
  \[
  \begin{tikzpicture}[xscale=6.5,yscale=-1.2]
    \node (A0_0) at (0, 0) {$\overline{M}_{0,n+2}(\mathbb{P}^{n},n)$};
    \node (A0_1) at (1, 0) {$\overline{M}_{0,n+2}\times(\mathbb{P}^{n})^n$};
    \node (A1_1) at (1, 1) {$(\mathbb{P}^{n})^n$};
    \path (A0_0) edge [->]node [auto] {$\scriptstyle{\rho\times ev_{1}\times...\times ev_{n+2}}$} (A0_1);
    \path (A0_1) edge [->]node [auto] {$\scriptstyle{\pi_2}$} (A1_1);
    \path (A0_0) edge [->]node [auto,swap] {$\scriptstyle{\pi}$} (A1_1);
  \end{tikzpicture}
  \]
gives a fibration $\pi$ of $\overline{M}_{0,n+2}(\mathbb{P}^{n},n)$ whose general fiber is isomorphic to $\overline{M}_{0,n+2}$. Since $\rho\times ev_{1}\times...\times ev_{n+2}$ is birational, and $\pi_2$ is a morphism with connected fibers between smooth varieties we have $\pi_{*}\mathcal{O}_{\overline{M}_{0,n+2}(\mathbb{P}^{n},n)}\cong\mathcal{O}_{(\mathbb{P}^{n})^n}$. Therefore, we may apply Theorem \ref{brion} with $G = \Aut^{o}(\overline{M}_{0,n+2}(\mathbb{P}^{n},n))$ and $f = \pi$. For any $\phi\in\Aut^{o}(\overline{M}_{0,n+2}(\mathbb{P}^{n},n))$ this yields an automorphism $\overline{\phi}\in \Aut^{o}((\mathbb{P}^{n})^n))$ such that the diagram 
 \[
  \begin{tikzpicture}[xscale=3.5,yscale=-1.2]
    \node (A0_0) at (0, 0) {$\overline{M}_{0,n+2}(\mathbb{P}^{n},n)$};
    \node (A0_1) at (1, 0) {$\overline{M}_{0,n+2}(\mathbb{P}^{n},n)$};
    \node (A1_0) at (0, 1) {$(\mathbb{P}^{n})^n$};
    \node (A1_1) at (1, 1) {$(\mathbb{P}^{n})^n$};
    \path (A0_0) edge [->]node [auto] {$\scriptstyle{\phi}$} (A0_1);
    \path (A1_0) edge [->]node [auto] {$\scriptstyle{\overline{\phi}}$} (A1_1);
    \path (A0_1) edge [->]node [auto] {$\scriptstyle{\pi}$} (A1_1);
    \path (A0_0) edge [->]node [auto,swap] {$\scriptstyle{\pi}$} (A1_0);
  \end{tikzpicture}
  \]
is commutative. Now, note that for any $(p_1,...,p_n)\in (\mathbb{P}^{n})^n$ there is a dense open subset of the fiber $F_{p_1,...,p_n} = \pi^{-1}(p_1,...,p_n)$ parametrizing rational normal curves in $\mathbb{P}^n$ through $p_1,...,p_n$. Let $x = (x,...,x)\in (\mathbb{P}^{n})^n$ be a point in the small diagonal $\Delta_{1,...,n}\cong\mathbb{P}^n$, and let $F_x = \pi^{-1}(x)$. Note that $y\in (\mathbb{P}^{n})^n\setminus\Delta_{1,...,n}$ forces $\dim(F_y)<\dim(F_x)$, where $F_y = \pi^{-1}(y)$. Therefore, $\overline{\phi}$ restricts to an automorphism of $\Delta_{1,...,n}\cong\mathbb{P}^n$, and we get a morphism of groups
$$
\begin{array}{cccc}
\chi: &\Aut^{o}(\overline{M}_{0,n+2}(\mathbb{P}^{n},n)) & \longrightarrow & PGL(n+1)\\
      & \phi & \longmapsto & \overline{\phi}
\end{array}
$$
Furthermore, by Section \ref{actions} $\chi$ is surjective. Finally we prove that $\chi$ is also injective. If $\overline{\phi} = \chi(\phi) = Id_{\mathbb{P}^n}$ then $\phi$ restricts to an automorphism of the general fiber of $\pi$. We know that such a general fiber is isomorphic to $\overline{M}_{0,n+2}$. Now, by \cite[Theorem 3]{BM} the automorphism group of $\overline{M}_{0,n+2}$ is isomorphic to the symmetric group $S_{n+2}$ for any $n\geq 3$.\\ 
This means that $\phi$ restricts to a permutation in $S_{n+2}$ on the general fiber of $\pi$, and since $\phi$ is in the connected component of the identity it must restrict to the trivial permutation on such a general fiber. Again to conclude it is enough to recall Remark \ref{flaut}. Finally, since $n\geq 3$ we have $n+2\geq 5$ and by Lemma \ref{induction} we conclude that $\Aut^{o}(\overline{M}_{0,k}(\mathbb{P}^n,n))\cong PGL(n+1)$ for $k\geq n+2$. 
\end{proof}

In the following we relate the automorphisms of the coarse moduli space $\overline{M}_{0,n}(\mathbb{P}^{N},d)$ with those of the stack $\overline{\mathcal{M}}_{0,n}(\mathbb{P}^{N},d)$.

\begin{Corollary}\label{corRNC}
For any $n\geq 3$ and $k\geq n+2$ we have $\Aut^{o}(\overline{\mathcal{M}}_{0,k}(\mathbb{P}^n,n))\cong PGL(n+1)$.
\end{Corollary}
\begin{proof}
Since $\overline{\mathcal{M}}_{0,k}(\mathbb{P}^n,n)$ is a smooth Deligne-Mumford stack with trivial general stabilizer it is enough to apply Proposition \ref{connRNC} and \cite[Proposition 1.7]{FaM}.
\end{proof}

\subsection{A conjecture on the automorphism group of $X[n]$}
From Proposition \ref{sym} it is clear that the diagonal action of $S_n\times \Aut(X)$ on $X^n$ lifts to the Fulton-MacPherson compactification $X[n]$. It is natural to ask whether this action gives the full automorphism group of $X[n]$.\\
By Proposition \ref{druel2} when $K_X$ is nef we may hope to control the automorphisms of $X[n]$. On the other hand, Remark \ref{GL2} shows that when $X$ is abelian we should expect the automorphisms of $X[n]$ to behave less nicely from our point of view.\\
Furthermore, by \cite{HMX} if $X$ is of general type then its group of birational automorphisms $\Bir(X)$ and a fortiori $\Aut(X)$ are finite. Therefore, we may expect the subgroup $\Aut_{\Delta}(X^n)\subseteq \Aut(X^n)$ in Proposition \ref{druel2} of automorphisms stabilizing the union of all the diagonals of codimension greater than one in $X^n$ to be just $S_n\times \Aut(X)$. This leads us to the following conjecture. 
\begin{Conjecture}\label{conj}
Let $X$ be a smooth projective variety of general type. If $n\neq 2$ then 
$$\Aut(X[n])\cong S_{n}\times \Aut(X)$$
\end{Conjecture}
Note that when $X = C$ is a curve this is just Proposition \ref{druel2}, and more generally when $X = C_1\times ...\times C_r$ is a product of curves with $g(C_i)\geq 2$ for any $i$ this is Proposition \ref{druel1}. On the other hand, the second part of Proposition \ref{druel1} tells us that additional symmetries are allowed when $n = 2$.

\bibliographystyle{amsalpha}
\bibliography{Biblio}
\end{document}